\documentclass[reqno, 11pt]{amsart}
\usepackage[colorlinks]{hyperref}
\usepackage{helvet}
\usepackage{mathrsfs}

\newtheorem{theorem}{Theorem}[section]
\newtheorem{lemma}[theorem]{Lemma}

\newtheorem{proposition}[theorem]{Proposition}

\theoremstyle{definition}

\newtheorem{remark}[theorem]{Remark}

\numberwithin{equation}{section}

\usepackage{bbm}
\usepackage{url}
\usepackage{euscript}
\usepackage{amsmath}
\usepackage{amsthm}

\usepackage{amsxtra}
\usepackage{amssymb}
\usepackage{pifont}
\usepackage{amsbsy}

\usepackage{graphicx}
\usepackage{epstopdf}

\newskip\aline \newskip\halfaline
\def\skipaline{\vskip\aline}

\def\qedbox{$\rlap{$\sqcap$}\sqcup$}
\def\qed{\nobreak\hfill\penalty250 \hbox{}\nobreak\hfill\qedbox\skipaline}

\def\proofend{\eqno{\mbox{\qedbox}}}

\newcommand{\one}{\mathbbm{1}}

\newcommand{\bN}{{{\mathbb N}}}

\newcommand\bR{{\mathbb R}}

\newcommand\bZ{{\mathbb Z}}

\DeclareMathOperator{\tr}{{\rm tr}}

 \DeclareMathOperator{\Hom}{Hom}

\DeclareMathOperator{\var}{\boldsymbol{var}}

\DeclareMathOperator{\GOE}{GOE}
\DeclareMathOperator{\vol}{vol}

\newcommand{\ba}{\boldsymbol{a}}

\newcommand{\be}{{\boldsymbol{e}}}

\newcommand{\ii}{\boldsymbol{i}}

\newcommand{\bp}{{\boldsymbol{p}}}

\newcommand{\bs}{{\boldsymbol{s}}}
\newcommand{\bt}{{\boldsymbol{t}}}
\newcommand{\bu}{{\boldsymbol{u}}}
\newcommand{\bv}{{\boldsymbol{v}}}

\newcommand{\bx}{{\boldsymbol{x}}}

\newcommand{\bsC}{\boldsymbol{C}}

\newcommand{\bsE}{\boldsymbol{E}}

\newcommand{\bsI}{\boldsymbol{I}}

\newcommand{\bsK}{{\boldsymbol{K}}}

\newcommand{\bsN}{\boldsymbol{N}}

\newcommand{\bsP}{\boldsymbol{P}}

\newcommand{\bsZ}{\boldsymbol{Z}}

\newcommand{\bgamma}{{\boldsymbol{\gamma}}}
\newcommand{\bGamma}{\boldsymbol{\Gamma}}

\newcommand{\blam}{{\boldsymbol{\lambda}}}

\newcommand{\si}{{\sigma}}

\newcommand{\ve}{{\varepsilon}}

\newcommand{\eF}{\EuScript{F}}

\newcommand{\eI}{{\EuScript{I}}}

\newcommand{\eL}{\EuScript{L}}

\newcommand{\eO}{\EuScript{O}}
\newcommand{\eP}{\EuScript{P}}

\newcommand{\eR}{\EuScript{R}}
\newcommand{\eS}{\EuScript{S}}

\newcommand{\eX}{\mathscr{X}}

\newcommand{\eZ}{\EuScript{Z}}

\newcommand{\ra}{\rightarrow}

\newcommand{\Llra}{{\Longleftrightarrow}}

\newcommand{\lan}{\langle}
\newcommand{\ran}{\rangle}

\def\inpr{\mathbin{\hbox to 6pt{\vrule height0.4pt width5pt depth0pt \kern-.4pt \vrule height6pt width0.4pt depth0pt\hss}}}

\newcommand{\pa}{\partial}

\newcommand{\ha}{\widehat{a}}

\newcommand{\whF}{\widehat{F}}

\begin{document}

\title[A central limit theorem]{A CLT concerning critical points of random functions on a Euclidean space}


\date{Started  August 28, 2015. Completed  on September 14, 2015. Last modified on {\today}. }

\author{Liviu I. Nicolaescu}

\address{Department of Mathematics, University of Notre Dame, Notre Dame, IN 46556-4618.}
\email{nicolaescu.1@nd.edu}
\urladdr{\url{http://www.nd.edu/~lnicolae/}}

\subjclass{60B20, 60D05, 60F05,  60G15}
\keywords{Gaussian random functions, critical points, Wiener chaos, Gaussian random matrices, central limit theorem}

\begin{abstract} We prove a central limit theorem concerning  the number of critical  points in large cubes of an isotropic  Gaussian random function on a Euclidean space.
\end{abstract}

\maketitle

\tableofcontents

\section{Introduction}  Throughout this paper $X(\bt)$ denotes a centered, isotropic Gaussian random function on $\bR^m$, $m\geq 2$.

Assume that $X$ is a.s. $C^1$. For  any Borel subset $S\subset \bR^m$ we denote  by $Z(S)$ the number of critical points  of $X$ in $S$.  For a positive number $L$ we set 
\[
Z_L:=Z\bigl(\,[-L,L]^m\,\bigr),
\]
 and we form the new random variable 
\begin{equation}\label{zeta}
\zeta_L:=\frac{1}{(2L)^{m/2}}\Bigl(\, Z_L-\bsE\bigl[\, Z_L\,\bigr]\,\Bigl).
\end{equation}
In this paper we will prove that, under certain assumptions on $X(\bt)$,  the sequence of random variables $(\zeta_N)$, converges in distribution as $N\to\infty$ to a Gaussian with mean zero and finite, positive variance.

The proof, inspired from the recent work of Estrade-Le\'{o}n \cite{EL},    uses  the Wiener   chaos decomposition  of $Z_L$.

\subsection*{Notations} 

\begin{itemize}

\item $\bN:=\bZ_{>0}$, $\bN_0:=\bZ_{\geq 0}$.

\item For any positive integer $n$ we denote by $\one_n$ the identity map $\bR^n\to\bR^n$.

\item For $\bt=(t_1,\dotsc, t_m)\in\bR^m$ we set 
\[
|\bt|:=\sqrt{\sum_{k=1}^m t_k^2},\;\;|\bt|_\infty:=\max_{1\leq k\leq m}|t_k|.
\]

\item  For any  $v\geq 0$ we denote by $\bgamma_v$ the   Gaussian measure on $\bR$ with mean $0$ and variance $v$.

\item If $X$ is a scalar random variable, then we will use the notation $X\in N(0,v)$ to indicate that  $X$ is a normal random variable with mean zero and variance $v$.

\item  If $A$ is a subset  of a  given set $S$, then we  denote by $\bsI_A$ the indicator function of $A$
\[
\bsI_A:S\to\{0,1\},\;\;\bsI_A(s)=\begin{cases}
1, & s\in A,\\
0, & s\in S\setminus A.
\end{cases}
\]
\item If $X$  is a random vector, then   we denote by $\bsE[X]$ and respectively $\var(X)$ the  mean and respectively the variance of $X$. 
\end{itemize}

\section{Statement of the main result}
\setcounter{equation}{0}

Denote by $K(\bt,\bs)$ the covariance kernel of $X(\bt)$, 
\[
K(\bs,\bt):=\bsE\bigl[\, X(\bt)X(\bs)\,\bigr],\;\;\bt,\bs\in\bR^m.
\]
The  isotropy   of $X$ implies that there exists a radially symmetric function $C:\bR^m\to \bR$ such that $K(\bt,\bs)=C(\bt-\bs)$, $\forall  \bt,\bs\in\bR$.  We denote by $\mu(d\blam)$ the spectral measure of $X$ so that $C(\bt)$ is the Fourier  transform of $\mu$
\begin{equation}\label{spectral}
C(\bt)=(2\pi)^{-\frac{m}{2}}\int_{\bR^m} e^{-\ii(\bt,\blam)}\mu(d\blam).
\end{equation}

\subsection{The setup} For the claimed  central limit result to hold, we need to make certain assumptions on the random function $X(t)$. These assumptions  closely mirror the assumptions in \cite{EL}.

\medskip

\noindent {\bf Assumption A1.}  \textit{The random function $X(t)$ is   almost surely $C^3$}.

\medskip

To formulate our next assumption we set
\begin{equation}\label{psi}
\psi(\bt):=\max\bigl\{\, |\pa^\alpha_\bt C(\bt)|;\;\;|\alpha|\leq 4\,\bigr\},\;\;\bt\in\bR^m,
\end{equation}
where for any multi-index $\alpha=(\alpha_1,\dotsc, \alpha_m)\in\bN_0^m$ we set
\[
|\alpha|:=\alpha_1+\cdots+\alpha_m,\;\;\pa^\alpha_\bt=\pa^{\alpha_1}_{t_1}\cdots \pa^{\alpha_m}_{t_m}.
\]
{\bf Assumption A2.} 
\[
\lim_{|\bt|\to\infty}\psi(\bt)=0\;\;\mbox{and}\;\; \psi\in L^1(\bR^m).
\]
Our next assumption involves the spectral measure $\mu(d\blam)$ and it states  in precise terms that this measure has a continuous density that decays   rapidly at $\infty$.

\medskip

\noindent {\bf Assumption A3.} \textit{There exists a nontrivial even, continuous function $w:\bR\to[0,\infty)$ such that
\[
\mu(d\blam) =w(|\blam|)d\blam.
\]
Moreover $w$  has a fast decay at $\infty$, i.e.,}
\[
|\lambda|^4 w(|\blam|)  \in L^1(\bR^m)\cap L^2(\bR^m). 
\]

\medskip

\begin{remark} (a) Let us observe that {\bf A1}-{\bf A3} imply that
\[
\psi\in L^q(\bR^m),\;\;\forall q>0.
\]
(b) The assumptions {\bf A1}-{\bf A3}  are automatically satisfied if the density $w$ is a Schwartz function on $\bR$.

\smallskip

\noindent (c) The paper \cite{EL}   includes one  extra assumption on $X$, namely that the Gaussian vector 
\[
J_2(X(0) ):=\bigl(\, X(0),\nabla X(0),\nabla^2 X(0)\,\bigr).
\]
is nondegenerate. We do not need this nondegeneracy in this paper, but  we want to mention that it is implied  by  Assumption {\bf A3}; see Proposition \ref{prop: j2}.   \qed
 \end{remark}

 Fix  real numbers $u\geq 0$ and $v>0$. Denote by  $\eS_m$ the space of real symmetric $m\times m$ matrices,   and by $\eS_m^{u,v}$ the space $\eS_m$     equipped with the centered Gaussian measure $\bGamma_{u,v}$ uniquely determined by the covariance equalities
 \[
 \bsE[a_{ij}a_{k\ell}]=  u\delta_{ij}\delta_{k\ell}+ v(\delta_{ik}\delta_{j\ell}+ \delta_{i\ell}\delta_{jk}),\;\;\forall 1\leq i,j,.k,\ell\leq m.
 \]
 In particular we have
 \begin{equation}\label{covmat}
 \bsE[a_{ii}^2]= u+2v,\;\;\bsE[a_{ii}a_{jj}]=u,\;\;\;\bsE[a_{ij}^2]=v,\;\;\forall 1\leq i\neq j\leq m,
 \end{equation}
while all other covariances are trivial.  The  ensemble  $\eS_m^{0,v}$ is    a rescaled version of  the Gaussian Orthogonal Ensemble  (GOE) and we will refer to it as $\GOE_m^v$.  As explained in  \cite{Ni2014, Ni2015},   the Gaussian measures  $\bGamma_{u,v}$  are  invariant with respect to the natural action of $O(m)$ on $\eS_m$. Moreover
\begin{equation}
d\bGamma_{0,v}(A)=(2v)^{-\frac{m(m+1)}{4}}\pi^{-^{\frac{m(m+1)}{4}} } \; e^{-\frac{1}{4v}\tr A^2} |dA|.
\label{eq: gov}
\end{equation} 

The ensemble $\eS_m^{u,v}$ can be given an alternate description.   More precisely   a random $A\in \eS_m^{u,v}$ can be described as a sum
\[
A= B+ \ X\one_m,\;\;B\in \GOE_m^v,\;\; X\in \bsN(0, u),\;\;\mbox{ $B$ and $X$ independent}.
\]
We write  this
\begin{equation}
\eS_m^{u,v} =\GOE_m^v\hat{+}\bsN(0,u)\one_m,
\label{eq: smgoe}
\end{equation}
where $\hat{+}$ indicates a sum of \emph{independent} variables. We set $\eS_m^v:= \eS_m^{v,v}$. Recall  from (\ref{spectral}) that
\[
\bsE\bigl[\,X(\bt)X(0)\,\bigr]=C(\bt)=K(\bt,0)=(2\pi)^{-\frac{m}{2}}\int_{\bR^m} e^{-\ii(\bt,\blam)} w(|\blam|) d\blam.
\]
Following \cite{Nispec} we define
\[ 
 s_m:=\frac{1}{(2\pi)^{m/2}}\int_{\bR^m} w(|\bx|) d\bx,\;\;d_m:= \frac{1}{(2\pi)^{m/2}}\int_{\bR^m} x_1^2w(|\bx|) d\bx,
 \]
 \[
 h_m:= \frac{1}{(2\pi)^{m/2}}\int_{\bR^m} x_1^2x_2^2w(|\bx|) d\bx.
 \]
Clearly $s_m,d_m,h_m>0$. If we set
 \begin{equation}
 I_k(w):=\int_0^\infty w(r) r^k dr,
 \label{eq: Ik}
 \end{equation}
 then we have  (see \cite{Nispec})
  \begin{equation}\label{sdh}
  \begin{split}
 (2\pi)^{m/2} s_m=\frac{2\pi^{\frac{m}{2}}}{ \Gamma(\frac{m}{2})} I_{m-1}(w),\;\;(2\pi)^{m/2} d_m= \frac{2\pi^{\frac{m}{2}}}{m \Gamma(\frac{m}{2})}I_{m+1}(w),\\
  (2\pi)^{m/2} h_m=\frac{1}{3}\int_{\bR^m}x_1^4 w(|x|) dx=\frac{2\pi^{\frac{m}{2}}}{m(m+2) \Gamma(\frac{m}{2})}I_{m+3}(w).
  \end{split}
  \end{equation}
  Then we deduce that 
  \begin{subequations}
  \begin{equation}\label{covx1}
  \bsE\bigl[\, X(0)\cdot\pa_{t_i}X(0)\,\bigr]=\bsE\bigl[\,\pa_{t_i}X(0)\cdot\pa^2_{t_jt_k}X(0)\,\bigr]=0,\;\;\forall i,j,k
  \end{equation}
  \begin{equation}\label{covx2}
  \bsE[X(0)^2]= s_m,\;\; \bsE\bigl[\,\pa_{t_i}X(0)\cdot \pa_{t_j}X(0)\,\bigr]=d_m\delta_{ij},\;\; \forall i,j,
  \end{equation}
  \begin{equation}\label{covx3}
  \bsE\bigl[\,X(0)\cdot\pa^2_{t_it_j}X(0)\,\bigr]=-d_m\delta_{ij},\;\;\forall i,j,
  \end{equation}
  \begin{equation}\label{covx4}
 \bsE\bigl[\,\pa^2_{t_it_j}X(0)\cdot\pa^2_{t_kt_\ell}X(0)\,\bigr]= h_m\bigl(\,\delta_{ij}\delta_{k\ell}+\delta_{ik}\delta_{j\ell}+ \delta_{i\ell}\delta_{jk}\bigr),\;\;\forall i,j,k,\ell.
 \end{equation}
 \end{subequations}
 The  equality (\ref{covx2}) shows that $\nabla X(0)$ is a $\bR^m$-valued  cenetered Gaussian random vector with covariance matrix $d_m\one_m$, while (\ref{covx4}) shows that  $\nabla^2 X(0)\in \eS_m^{h_m}$. \qed

\subsection{The main result} We can now state the main result of this paper.

\begin{theorem}\label{th: main} Suppose that $X(\bt)$ is a  centered, stationary, isotropic random function on $\bR^m$, $m\geq 2$ satisfying assumptions {\bf A1}, {\bf A2}, {\bf A3}. Denote  by $Z_N$ the number of critical points of $X(\bt)$ in the cube $C_N:=[-N,N]^m$. Then the following hold.

\begin{enumerate}

\item 
\begin{equation}\label{expect}
\bsE\bigl[\, Z_N\,\bigr]= C_m(w) (2N)^m,\;\;\forall N,
\end{equation}
where
\begin{equation}\label{cmw}
C_m(w)= \left(\frac{h_m}{2\pi d_m}\right)^{\frac{m}{2}}\bsE_{\eS_m^1}\bigl[\, |\det A|\,\bigr].
\end{equation}
\item There exists  a constant $V_\infty=V_\infty(m,w)>0$ such that
\begin{equation}\label{varz}
\var(Z_N)\sim V_\infty N^m\;\;\mbox{as $N\to\infty$}.
\end{equation}
Moreover, the sequence    of random variables 
\[
\zeta_N=N^{-m/2}\Bigl( \, Z_N-\bsE\bigl[\, Z_N\,\bigr]\,\Bigr)
\]
converges in law to a normal  random variable $\zeta_\infty$ with mean zero and positive variance $V_\infty$.

\end{enumerate}

\end{theorem}

\begin{remark} (a) The isotropy condition on  $X(\bt)$  may be a bit restrictive,  but we  believe that the techniques in \cite{EL} and this paper extend to  the  more general case of stationary  random functions. (The only place where isotropy  plays an crucial role is in the proof of Proposition \ref{prop: zl2}.)  However, for the   geometric applications we have in mind, the  isotropy   is a natural assumption. Let us elaborate on this point.

Suppose that $(M,g)$ is a compact $m$-dimensional Riemannian manifold,  such that $\vol_g(M)=1$. Denote by $\rho$ the injectivity radius of $g$. For $\ve>0$  we denote by $g_\ve$ the rescaled metric  $g_\ve:= \ve^{-2} g$. Intuitively,  as $\ve\to 0$, the metric $g_\ve$ becomes flatter and flatter.  Denote by $\Delta_g$ the Laplacian  of $g$  and by  $\Delta_{\ve}$ of  $g_\ve$.  Let 
\[
\lambda_0\leq \lambda_1\leq \lambda_2\leq \cdots
\]
be the eigenvalues of $\Delta_g$, multiplicities included. Fix an orthonormal basis of $L^2(M,dV_g)$ consisting of eigenfunctions  $\Psi_k$ of $\Delta_g$,
\[
\Delta_g \Psi_k=\lambda_k\Psi_k
\]
Then the eigenvalues of $\Delta_\ve$ are $\lambda_k(\ve)=\ve^2\lambda_k$ with corresponding eigenfunctions  $\Psi_k^\ve:=\ve^{\frac{m}{2}}\Psi_k$. 

We now define the random function $Y_\ve(\bp)$ on $M$,
\[
Y_\ve(\bp)=\sum_{k=0}^\infty w\bigl(\,\sqrt{\lambda_k(\ve)}\,\bigr)^{\frac{1}{2}}Z_k \Psi_k^\ve(\bp),
\]
where $(Z_k)_{k\geq 0}$ is a sequence of independent standard normal random variables.

Fix a point $\bp_0\in M$ and denote by $\exp_\ve$ the exponential map $T_{\bp_0} M\to M$ defined by the metric $g_\ve$.  (This is a diffeomorphism onto when restricted  to the ball of radius $\ve^{-1}\rho$ of the tangent space $T_{\bp_0}M$ equipped with the metric $g_\ve$.)  Denote by $\eR_\ve$ the rescaling map
\[
T_{\bp_0}M\to T_{\bp_0}M,\;\; \bv\mapsto\ve  \bv.
\]
This map is an isometry   $(T_{\bp_0}M,g)\to (T_{\bp_0} M,g_\ve)$. We denote by $X_\ve$ the random function on  the Euclidean space $(T_{\bp_0} M,g)$ obtained   by pulling back $Y_\ve$ via the map $\exp_\ve\circ  \eR_\ve$.  The  random function $X_\ve(t)$ is Gaussian and  its covariance kernel converges in the $C^\infty$ topology   as $\ve \to 0$ to the covariance kernel of  random function $X$ we are investigating in this paper.   Denote by  $N(X_\be, B_r)$ the  number critical points  of $Y_\ve$ in a  $g$-ball of  radius $r<\rho$ on $M$, and by $N(X,B_ R)$ the number of critical points of $X$ in  a ball of radius $R$.  In \cite{Nispec} we have shown that
\[
\bsE\bigl[\, N(Y_\ve, B_r)\,\bigr]\sim \bsE[ \, N(X,B_{r/\ve}\,\bigr]= const.\; \ve^{-m}\;\;\mbox{as $\ve\to 0$}.
\]
Additionally, in \cite{Nivar} we looked at the special case when $M$ is a flat $m$-dimensional torus  and we  showed that the  variances random variables
\[
\ve^{-m/2}\Bigl(\, N(Y_\ve, B_r)-\bsE\bigl[\, N(Y_\ve, B_r)\,\bigr]\,\Bigr),\;\;\ve^{-m/2}\Bigl(\, N(X, B_{r/\ve})-\bsE\bigl[\, N(X, B_{r/\ve})\,\bigr]\,\Bigr)
\]
have the identical  \emph{finite} limits  as $\ve>0$. In \cite{Nivar} we were not able to prove that this common limit is nonzero, but Theorem \ref{th: main} shows this to be the case.  

These facts suggest that the random variable $N(Y_\ve, B_r)$ may satisfy a central limit theorem of the type proved  in \cite{AL, GW}.  We will pursue this line of investigation elsewhere. \qed
\end{remark}

\subsection{Outline of the proof.}\label{ss: 23}  The strategy of proof  owes a great deal to \cite{EL}.  The  Gaussian random variables  $X(\bt)$, $\bt\in\bR^m$, are defined on  the same probability space $(\Omega, \eF, \bsP)$. We denote by $\eX$ the $L^2$-closure of the span of these variables. Thus, $\eX$ is Gaussian Hilbert space in the sense of \cite{Jan, Maj}, and  we  have a white-noise Hilbert space  isomorphism
\[
W: L^2\bigl(\, \bR^m, \sqrt{w(\blam)}\,d\blam\,\bigr)\to \eX.
\]
  We set $\widehat{\eX}:=L^2(\Omega, \widehat{\eF},\bsP)$, where  $\widehat{\eF}\subset \eF$ is the $\si$-algebra generated by the variables $X(\bt)$, $\bt\in\bR^m$. 

Using the Kac-Rice formula we  show in  Subsections \ref{ss: 31} and \ref{ss: 32}     that $Z_N\in\widehat{\eX}$  and we  describe  its  Wiener chaos decomposition.    In particular, this leads to a proof of (\ref{expect}).
To prove that the random variables
\[
\zeta_N= N^{-m/2}\bigl(\, Z_N- \bsE[Z_N]\,\bigr)
\]
converge in law to a normal  random variable we rely  on the    very general Breuer-Major type  central limit  theorem \cite[Thm. 6.3.1]{NP}, \cite{NPP}.     Here are the details.  

Denote by $\zeta_{N,q}$ the $q$-th chaos component   of $\zeta_N\in\widehat{\eX}$.   According to \cite[Thm.6.3.1]{NP},  to prove  that $\zeta_N$ converges in law  as $N\to\infty$ to a  normal random variable  $\zeta_\infty$ it suffices to prove the following.

\begin{enumerate}

\item  For every $q\in\bN$ there exists $V_{q,\infty}\geq 0$ such that
\[
\lim_{N\to\infty} \var(\zeta_{N,q})=V_{q,\infty}.
\]
\item 
\[
\lim_{Q\to\infty} \sup_N \sum_{q\geq Q} \var(\zeta_{N,q}) =0.
\]
\item  For   each $q\in\bN$, the random variables $\zeta_{N,q}$ converge in law  as $N\to\infty$   to a normal random variable, necessarily of variance $V_{q,\infty}$.
\end{enumerate}

We prove (i) and (ii) in the first half of  Subsection \ref{ss: 33}; see  (\ref{vq2}) and respectively Lemma \ref{lemma: VQ}. In the second half of this subsection we prove  that $V_{2,\infty}>0$.     

To prove   (iii) we rely on the fourth-moment theorem \cite[Thm. 5.2.7]{NP}, \cite{NuaPe}. The details are identical to the ones employed in the proof of  \cite[Prop. 2.4]{EL}.

The variance  of the  limiting normal random variable $N_\infty$ is
\[
V_\infty=\sum_{q\geq 2} V_{q,\infty}  <\infty.
\]
Appendix \ref{s: a} contains estimates  of the lower order terms in  the Hermite polynomial decomposition of $|\det A|$ where $A\in \eS_m^v$, $m\gg 1$.  These  estimates can be used to produce explicit positive lower bounds  for $V_\infty$ for large $m$.

\subsection{Related work}  Central limit  theorems concerning   counts of zeros of random functions go back a while, e.g. T. Malevich \cite{Mal69} (1969) and J. Cuzik \cite{Cuz76} (1976).

  The usage of  Wiener chaos decompositions  and of  Breuer-Major type results  in proving  such  central limit  theorems is more recent, late 80s early 90s and we want to mention here the pioneering  contributions of  Chambers and Slud \cite{ChaSl},   Slud \cite{Slud91, Slud94},  Kratz and Le\'{o}n \cite{KL1997} and Sodin-Tsirelson \cite{ST}.

 This topic  was further elaborated by  Kratz  and Le\'{o}n  in \cite{KL2001}  where they  also proved a central limit theorem  concerning the length of the zero set of a random function of two variables. We refer to \cite{AzWs}  for  a particularly nice presentation of these  developments.   
 
  Aza\"{i}s and Le\'{o}n \cite{AL} used  the technique of  Wiener  chaos decomposition to give a  shorter and more conceptual proof   to a central limit theorem   due to Granville and Wigman \cite{GW}  concerning the number  of zeros of random trigonometric polynomials of large degree. This technique was then succesfully   used by  Aza\"{i}s, Dalmao and Le\'{o}n in \cite{AzDaLe}  to prove a CLT concerning the number of zeros of  Gaussian   even trigonometric polynomials and by Dalmao in \cite{Dal} to  prove a CLT concerning the number of zeros of    usual polynomials   in the Shub-Smale probabilistic ensemble. Recently, Adler and Naitzat \cite{AN} used  Hermite  decompositions  to  prove a CLT  concerning Euler integrals of random functions.

 \subsection*{Acknowledgments} I want to thank Yan Fyodorov for  sharing with me the tricks in Lemma \ref{lemma: fy}.

\section{Proof of the main result}
\setcounter{equation}{0}

The random variables $X(\bt)$, $\bt\in\bR^m$ are defined on a common  probability space $(\Omega,\eO, \bsP)$.   We denote by $\eO_X$ the $\si$-subalgebra of $\eO$ generated by the variables $X(t)$, $t\in\bR^m$. For simplicity we set $L^2(\Omega):=L^2(\Omega, \eO_X,\bsP)$.

As detailed in  e.g. \cite{Jan, Maj, Nua}, the space  $L^2(\Omega)$ admits an orthogonal  Wiener chaos decomposition
\[
L^2(\Omega)=\overline{\bigoplus_{q=0}^\infty L^2(\Omega)_q},
\]
where $L^2(\Omega)_q$ denotes the  $q$-th chaos component. We let  $\eP_q:L^2(\Omega)\to L^2(\Omega)_q$ denote the  orthogonal projection on the $q$-th chaos component.

Let $T$ denote a compact parallelipiped 
\[
T:=[a_1,b_1]\times \cdots \times [a_m,b_m],\;\; a_i<b_i,\;\;\forall i=1,\dotsc, m.
\]
From \cite[Thm.11.3.1]{AT}, we deduce that $X$  is a.s. a Morse function on $T$. In particular,  almost surely there are no critical points on the boundary of $T$. 

\subsection{Chaos decompositions of  functionals of  random symmetric matrices.}\label{ss: 31} The dual space $(\eS_m^v)^*=\Hom(\eS_m^v,\bR)$ is a finite dimensional Gaussian  linear space in the sense of \cite{Jan} spanned by the entries  $(a_{ij})_{1\leq i\leq j\leq m}$ of a random matrix $A\in\eS_m^v$.  Its Fock space is  the space $L^2 (\eS_m, \bGamma_{v,v})$  and admits an orthogonal  chaos decomposition,
\[
L^2(\eS_m^v)=\overline{\bigoplus_{q\geq 0} L^2(\eS_m^v)_q}.  
\]
We recall that 
\[
\widehat{\eP}_{q,v}:=\bigoplus_{k=0}^q L^2(\eS_m^v)_k
\]
is the subspace of $L^2(\eS^v_m)$ spanned by  polynomials  of degree $\leq n$ in the entries of $A\in\eS^v_m$, and $L^2(\eS_m^v)_q$ is the orthogonal complement of $\widehat{\eP}_{q-1,v}$ in $\widehat{\eP}_{q,v}$.   The  summand $L^2(\eS_m^v)_q$ is called the \emph{$q$-th chaos component} of $L^2(\eS_m^v,)$.  

The chaos  decomposition construction is equivariant with respect to the action of $O(m)$ on $\eS_m^{v}$. In particular, the chaos components $L^2(\eS_m^v)_k$ are $O(m)$-invariant subspaces. If we denote by $L^2(\eS_m^v)^{inv}$ the subspace  of $L^2(\eS_m^v)$  consisting of $O(m)$-invariant functions, then   we obtain an orthogonal decomposition
\begin{equation}\label{chaos}
L^2\bigl(\, \eS_m^v\,\bigr)^{inv}=\overline{\bigoplus_{k\geq 0} L^2(\eS_m^v)_k^{inv}},
\end{equation}
where $L^2(\eS_m^v)_k^{inv}$ consists of the subspace of $L^2\bigl(\, \eS_m^v\,\bigr)_k$ where $O(m)$ acts trivially. In particular, we deduce that  $L^2(\eS_m^v)_k^{inv}$ consists of polynomials in the variables  $\tr A,\tr A^2,\dotsc,\tr A^m$.

We define the $O(m)$-invariant functions
\[
p,q, f:\eS_m\to\bR,\;\;p(A)=(\tr A)^2,\;\;q(A):=\tr A^2,\;\;f(A)=|\det A|.
\]
A basis of  $\widehat{\eP}_{2,v}^{inv}$ is given by the polynomials $1$, $\tr A$,  $p(A)$, $q(A)$. Clearly, since  $\tr A$ is an odd function of $A$, it is  orthogonal to the even polynomials $1,p(A),q(A)$. We have
\begin{equation}\label{1f1}
\bsE_{\eS_m^v}\bigl[\,p(A) \,\bigr]=\sum_{i=1}^n \bsE[a_{ii}^2]+2\sum_{i<j}\bsE[a_{ii}a_{jj}]=3mv+m(m-1)v=m(m+2)v.
\end{equation}
We have
\begin{equation}\label{1f2}
\bsE_{\eS_m^v}[q(A)]=\sum_{i=1}^m \bsE[a_{ii}^2]+2\sum_{i<j}\bsE[a_{ij}^2]= 3m+ m(m-1)=m(m+2)v.
\end{equation}
We deduce that the polynomials
\begin{equation}\label{barpq}
\begin{split}
\bar{p}(A)=p(A)-\bsE_{\eS_m^v}[p(A)]= p(A)-m(m+2)v,\\
\bar{q}(A)=q(A)-\bsE_{\eS_m^v}[q_A)]=q(A)-m(m+2)v
\end{split}
\end{equation}
form a (non-orthonormal)  basis of  $L^2(\eS_m^v)_2^{inv}$.   

\subsection{Hermite decomposition of $Z(T)$. }\label{ss: 32}  For $\ve>0$ define
\[
\delta_\ve:\bR^m\to\bR,\;\;\delta_\ve=(2\ve)^{-m}\bsI_{[-\ve,\ve]^m}.
\]
Observe that  the family $(\delta_{\ve})$ approximates  the Dirac distribution $\delta_0$ on $\bR^m$ as $\ve\searrow 0$. We recall  \cite[Prop. 1.2]{EL}  which applies with  no change to the setup in this  paper.

\begin{proposition}[Estrade-Le\'{o}n]\label{prop: zl2}  The random variable $Z(T)$ belongs to $L^2(\Omega)$. Moreover
\[
Z(T)=\lim_{\ve\searrow 0} \int_T \bigl|\,\det \nabla^2 X(\bt)\,\bigr| \delta_{\ve}\bigl(\,\nabla X(\bt)\,\bigr) d\bt
\]
almost surely and in $L^2(\Omega)$.\qed
\end{proposition}

The above nontrivial  result implies that the random variable $Z(T)$ has  finite variance and  admits a  chaos decomposition  as elaborated for example in \cite{Jan, Maj, Nua}. 

Recall that an orthogonal basis of $L^2\bigl(\, \bR,\bgamma_1(dx)\,\bigr)$ is given   by  the Hermite polynomials, \cite[Ex. 3.18]{Jan},  \cite[V.1.3]{Mal},
\begin{equation}\label{hermn}
H_n(x) :=(-1)^n e^{\frac{x^2}{2}} \frac{d^n}{dx^n}\Bigl(\, e^{-\frac{x^2}{2}}\,\Bigr)= n!\sum_{r=0}^{\lfloor\frac{n}{2}\rfloor} \frac{(-1)^r}{2^r r!(n-2r)!} x^{n-2r}.
\end{equation}
In particular
\begin{equation}\label{hn0}
H_n(0)=\begin{cases}
0, & n\equiv 1\bmod 2,\\
(-1)^r\frac{(2r)!}{2^rr!}, & n=2r.
\end{cases}
\end{equation}
For every multi-index  $\alpha=(\alpha_1, \alpha_2,\dotsc )\in\bN_0^{\bN}$ such that    all but finitely many $\alpha_k$-s are nonzero, and any
\[
\underline{x}=(x_1,x_2,\dotsc)\in\bR^{\bN}
\]
 we set
\[
|\alpha|:=\sum_k\alpha_k,\;\;\alpha!:=\prod_k \alpha_k!,\;\;H_\alpha(\underline{x}):=\prod_k H_{\alpha_k}(x_k).
\]
To simplify the notation  we set
\[
U(\bt):=\frac{1}{\sqrt{d_m}}\nabla X(\bt),\;\;A(\bt):=\nabla^2 X(\bt).
\]
Thus $U(\bt)$ is  a $\bR^m$-valued  Gaussian random vector  with covariance  matrix $\one_m$ while $A(t)$ is a Gaussian  random symmetric matrix in the ensemble $\eS_m^{h_m}$.

Recall that  $f(A)=|\det A|$. We have $f\in L^2(\eS_m^{h_m})^{inv}$ and we denote by $f_n(A)$ the component of  $f(A)$ in the $n$-th chaos  summand of the chaos decomposition (\ref{chaos}). Since $f$ is an even function we deduce that $f_n(A)=0$ for odd $n$.  Note also that
\[
f_0(A)=\bsE_{\eS_m^{h_m}}\bigl[\,|\det A|\,\bigr]\neq 0.
\]
Following \cite[Eq.(5)]{EL}   we define for every $\alpha\in\bN_0^m$ the quantity
\begin{equation}\label{dalph}
d(\alpha):=\frac{1}{\alpha!}(2\pi d_m)^{-\frac{m}{2}}  H_\alpha(0).
\end{equation}
Arguing exactly as in the proof of \cite[Prop. 1.3]{EL} we deduce the following result.
\begin{proposition}\label{prop: el1.3}
The following expansion holds in $L^2(\Omega)$
\[
Z(T)=\sum_{q=0}^\infty Z_q(T),
\]
where
\[
Z_q(T)=P_qZ(T)=\sum_{\substack{\alpha\in \bN_0^m, \;n\in\bN_0,\\|\alpha|+n=q}}d(\alpha)\int_T H_\alpha\bigl(\, U(\bt)\,\bigr)f_n\bigl(\, A(\bt)\,\bigr) d\bt. \proofend
\]
\end{proposition}

 Observe that the expected number of critical points of $X$ on $T$ is given by
\[
\bsE\bigl[\, Z(T)\,\bigr]=  \bsE\bigl[\, Z(T)\,\bigr] =d(0)\int_T \bsE\bigl[ H_0\bigl(\,U(\bt)\,\bigr) f_0\bigl(\,A(\bt)\,\bigr)\,\bigr] d\bt
\]
(use the stationarity of $X(\bt)$)
\[
=(2\pi d_m)^{-\frac{m}{2}} \int_T \bsE_{\eS_m^{h_m}}\bigl[\, f(A(\bt))\,\bigr] d\bt =(2\pi d_m)^{-\frac{m}{2}}\bsE_{\eS_m^{h_m}}\bigl[\,|\det A|\,\bigr]\vol(T)
\]
\[
=\left(\frac{h_m}{2\pi d_m}\right)^{\frac{m}{2}}\bsE_{\eS_m^1}\bigl[\, |\det A|\,\bigr]\vol(T).
\]
This proves (\ref{expect}) and (\ref{cmw}).

\subsection{Asymptotic variance of $Z(T)$.}\label{ss: 33}   Denote by $C_N$ the cube $[-N,N]^m$ and  by $B$ the cube $[0,1]^m$. For any Borel measurable subset $S\subset \bR^m$ such that $\vol (S)\neq 0$ we set
\[
\zeta(S):=\frac{1}{\sqrt{\vol(S)}}\bigl(\, Z(S)-\bsE\bigl[\, Z(S)\,\bigr]\,\bigr).
\]
Thus, $\zeta_N=\zeta(C_N)$. Since $\zeta_N\in L^2(\Omega)$ we deduce $\var(\zeta_N)<\infty$. 

\begin{proposition}\label{prop: el21}  There exists $V_\infty\in (0,\infty)$ such that
\[
\lim_{N\to\infty}\var(\zeta_N) =V_\infty.
\]
\end{proposition}

\begin{proof} To prove that the above limit exists we follow closely the  proof of \cite[Prop. 2.1]{EL}. We have
\[
\zeta_N=\zeta(C_N)=(2N)^{-\frac{m}{2}}\sum_{q\geq 1} Z_q(C_N),
\]
\[
V_N:=\var\bigl(\,\zeta(C_N)\,\bigr)=\sum_{q\geq 1}\; \underbrace{(2N)^{-m}\bsE\bigl[ Z_q(C_N)^2\,\bigr]}_{=:V_{q,N}}.
\]
To estimate $V_{q,N}$   we write
\[
Z_q(T)=\int_T \rho_q(\bt) d\bt,
\]
where
\begin{equation}\label{rhoq}
\rho_q(\bt)= \sum_{\substack{\alpha\in \bN_0^m, \,n\in\bN_0\\|\alpha|+n=q}}d(\alpha) H_\alpha\bigl(\, U(\bt)\,\bigr)f_n\bigl(\, A(\bt)\,\bigr).
\end{equation}
Then
\[
V_{q,N}=(2N)^{-m}\int_{C_N\times C_N} \bsE\bigl[ \,\rho_q(\bs)\rho_q(\bt)\,\bigr] d\bs d\bt
\]
(use the stationarity of $X(t)$)
\[
=(2N)^{-m}\int_{C_N\times C_N} \bsE\bigl[ \,\rho_q(0)\rho_q(\bt-\bs)\,\bigr] d\bs d\bt
\]
\[
=(2N)^{-m}\int_{C_{2N}} \bsE\bigl[ \,\rho_q(0)\rho_q(\bu)\,\bigr]\prod_{k=1}^m\left(1-\frac{|u_k|}{2N}\right) d\bu.
\]
The last equality is  obtained by integrating along the fibers of the  map 
\[
C_N\times C_N\ni (\bs,\bt)\mapsto \bt-\bs\in C_{2N}.
\]
To estimate the last integral, we fix an orthonormal basis  $(b_{ij})_{1\leq i\leq j\leq m}$ of the Gaussian Hilbert  space $\Hom(\eS_n^{h_m},\bR)$.   We denote   by $B$ the vector  $(b_{ij})_{1\leq i\leq j\leq m}$, by $A$ the vector  $(a_{ij})_{1\leq i\leq j\leq m}$ both viewed as column vectors of size 
\[
\nu(m):=\dim\eS_m=\frac{m(m+1)}{2}.
\]
There exists a nondegenerate  deterministic matrix $\Lambda$ of size $\nu(m)\times \nu(m)$, relating $A$ and  $B$, $A=\Lambda B$. We now have  an orthogonal decomposition
\[
f_n(A)=\sum_{\beta\in \bN_0^{\nu(m)},\;|\beta|=n} c(\beta)H_\beta(B).
\]
Let us set
\[
\eI_m:=\bigl(\bN_0^m\bigr)\times  \bigl( \bN_0^{\nu(m)}\bigr).
\]
We deduce
\[
\rho_q(\bt)= \sum_{(\alpha,\beta)\in\eI_m} d(\alpha)c(\beta) H_\alpha(\,U(t)\,) c(\beta)H_\beta(\,B(\bt)\,).
\]
We can further simplify  this formula   if we introduce the vector 
\[
Y(\bt):=\bigl(\,U(\bt), B(\bt)\,\bigr),\;\;B(t)=\Lambda^{-1}A(\bt).
\]
 For $\gamma=(\alpha,\beta)\in \eI_m$ we set 
\[
\ba(\gamma):=d(\alpha)c(\beta),\;\;H_\gamma\bigl(\,Y(\bt)\,\bigr):=H_\alpha\bigl(\, U(\bt)\,\bigr)H_\beta\bigl(\,B(\bt)\,\bigr).
\]
 Then
\begin{equation}\label{rhoq0}
\rho_q(\bt)=\sum_{\gamma\in\eI_m,\;|\gamma|=q} \ba(\gamma) H_\gamma\bigl(\,Y(\bt)\,\bigr),
\end{equation}
\[
\bsE\bigl[ \,\rho_q(0)\rho_q(u)\,\bigr]=\sum_{\substack{\gamma,\gamma'\in\eI_m\\ |\gamma|=|\gamma'|=q}} \ba(\gamma)\ba(\gamma')\bsE\bigl[\,H_\gamma(\,Y(0)\,) H_{\gamma'}(\,Y(\bu)\,)\,\bigr].
\]
We set  $\omega(m):=m+\nu(m)$, and we denote by $Y_i(\bt)$, $1\leq i\leq  \omega(m)$, the components of $Y(\bt)$ labelled so  that $Y_i(\bt)=U_i(\bt)$, $\forall 1\leq i\leq m$.  For $\bu\in\bR^m$  and $1\leq i,j\leq \omega(m)$ we define the covariances
\[
\Gamma_{ij}(\bu):=\bsE\bigl[\,Y_i(0)Y_j(\bu)\,\bigr].
\]
Observe that there exists a positive constant $K$ such that
\begin{equation}\label{psik}
\bigl|\,\Gamma_{i,j}(\bu)\,\bigr|\leq K\psi(\bu),\;\;\forall i,j=1,\dotsc, \omega(m),\;\;\bu\in\bR^m,
\end{equation}
where $\psi$ is the function defined in (\ref{psi}).

Using the Diagram Formula (see e.g.\cite[Cor. 5.5]{Maj} or \cite[Thm. 7.33]{Jan}) we deduce that  for any $\gamma,\gamma'\in\eI_m$ such that $|\gamma|=|\gamma'|=q$ there exists a universal homogeneous polynomial of degree $q$, $P_{\gamma,\gamma'}$ in the variables $\Gamma_{ij}(\bu)$ such that
\[
\bsE\bigl[\,H_\gamma(\,Y(0)\,) H_{\gamma'}(\,Y(\bu)\,)\,\bigr]=P_{\gamma,\gamma'}\bigl(\,\Gamma_{ij}(\bu)\,\bigr).
\]
Hence
\begin{equation}\label{vqn1}
V_{q,N}=(2N)^{-m}\sum_{\substack{\gamma,\gamma'\in\eI_m\\
|\gamma|=|\gamma'|=q}} \ba(\gamma)\ba(\gamma')\;\underbrace{\int_{C_{2N}}P_{\gamma,\gamma'}\bigl(\,\Gamma_{ij}(\bu)\,\bigr) \prod_{k=1}^m\left(1-\frac{|u_k|}{2N}\right) d\bu}_{=:R_N(\bgamma,\bgamma')}.
\end{equation}

From (\ref{psik}) we deduce that   for any $\gamma,\gamma'\in\eI_m$ such that $|\gamma|=|\gamma'|=q$ there exists a constant $C_{\gamma,\gamma'}>0$ such that
\[
\bigl|\, P_{\gamma,\gamma'}\bigl(\,\Gamma_{ij}(\bu)\,\bigr) \,\bigr|\leq C_{\gamma,\gamma'}\psi(\bu)^q,\;\;\forall \bu\in\bR^m.
\]
Since $\psi\in L^p(\bR^m)$, $\forall p>1$, we deduce   from the dominated convergence theorem that
\begin{equation}\label{rnlim}
\lim_{N\to\infty} R_N(\bgamma,\bgamma')=R_\infty(\gamma,\gamma'):= \int_{\bR^m}P_{\gamma,\gamma'}\bigl(\,\Gamma_{ij}(\bu)\,\bigr) d\bu, 
\end{equation}
and thus
\begin{equation}\label{vq2}
\lim_{N\to\infty} V_{q,N}=V_{q,\infty}:= \sum_{\substack{\gamma,\gamma'\in\eI_m\\
|\gamma|=|\gamma'|=q}} \ba(\gamma)\ba(\gamma') R_\infty(\gamma,\gamma')=\int_{\bR^m} \bsE\bigl[ \,\rho_q(0)\rho_q(\bu)\,\bigr]d\bu.
\end{equation}
Since $V_{q,N}\geq 0$, $\forall q, N$, we have
\[
V_{q,\infty}\geq 0,\;\;\forall q.
\]
\begin{lemma}\label{lemma: VQ} For any positive integer $Q$ we set
\[
V_{>Q, N}:=\sum_{q>Q}V_{q,N}.
\]
Then
\begin{equation}\label{supN}
\lim_{Q\to\infty} \bigl(\,\sup_NV_{>Q, N}\,\bigr) =0,
\end{equation}
the series 
\[
\sum_{q\geq 1}V_{q,\infty}
\]
is convergent and, if $V_\infty$ is its sum, then
\begin{equation}\label{vinfty}
V_\infty=\lim_{N\to\infty} V_N=\lim_{N\to\infty}\sum_{g\geq 1}  V_{q, N}.
\end{equation}
\end{lemma}

\begin{proof} For $\bs\in\bR^m$ we denote by $\theta_\bs$ the shift operator associated with the field  $X$, i.e.,
\[
\theta_\bs X(\bullet)= X(\bullet+\bs).
\] 
This extends to a unitary map $L^2(\Omega)\to L^2(\Omega)$  that commutes with the chaos decomposition of $L^2(\Omega)$. Moreover,  for any  parallelipiped $T$ we have
\[
Z(T+\bs)=\theta_\bs Z(T).
\]
If we denote by $\eL_N$ the set of lattice points
\[
\eL_N:=[-N,N)^m\cap\bZ^m
\]
then we deduce
\[
\zeta(C_N)=(2N)^{-m/2}\sum_{\bs\in\eL_m} \theta_\bs\zeta\bigl(B),\;\;B=[0,1]^m.
\]
We denote by $\eP_{>Q}$ the projection
\[
\eP_{>Q}=\sum_{q>Q}\eP_q,
\]
where we recall that $\eP_q$ denotes the projection on the $q$-th chaos component of $L^2(\Omega)$. We have
\[
\eP_{>Q}\zeta(C_N)=(2N)^{-m/2}\sum_{\bs\in\eL_m} \theta_\bs \eP_{>Q}\zeta\bigl(B).
\]
Using the stationarity of $X$ we deduce
\begin{equation}\label{V>Q}
V_{>Q, N}=\bsE\Bigl[\,\bigr|\eP_{>Q}\zeta(C_N)\bigr|^2\,\Bigr]=(2N)^{-m}\sum_{\bs\in\eL_{2N}}\nu(\bs, N)\bsE\bigl[\,\eP_{>Q}\zeta\bigl(B)\cdot\theta_\bs \eP_{>Q}\zeta\bigl(B)\,\bigr],
\end{equation}
where $\nu(\bs,N)$ denotes the number of lattice points $\bt\in\eL_N$ such that $\bt-\bs\in\eL_N$. Clearly
\begin{equation}\label{nubs}
\nu(\bs, N)\leq (2N)^m.
\end{equation}
With $K$ denoting the positive constant in (\ref{psik})  we choose  positive numbers $a,\rho$ such that
\[
\psi(\bs)\leq \rho<\frac{1}{K},\;\;\forall |\bs|_{\infty}>a.
\]
We split $V_{>Q,N}$ into two parts
\[
V_{>Q,N}=V'_{>Q,N}+V''_{>Q,N},
\]
where  $V'_{>Q,N}$ is made up of  the terms in (\ref{V>Q})  corresponding to  lattice points $\bs\in\eL_{2N}$ such that $|\bs|_\infty <a+1$,  while $V''_{>Q,N}$  corresponds to lattice points $\bs\in\eL_{2N}$ such that  $|\bs|_\infty\geq a+1$.

We  deduce from (\ref{nubs}) that for $2M>a+1$ we have
\[
\Bigr| \,V'_{>Q,N}\,\Bigr|\leq (2N)^{-m}(2a+2)^m  (2N)^m\bsE\Bigl[\, \bigl|\eP_{>Q}\zeta(B)\bigr|^2\,\Bigr].
\]
As $Q\to \infty$, the right-hand side of the above inequality goes to $0$ uniformly with respect to $N$.

To estimate $V''_{>Q,N}$ observe that for $\bs\in \eL_{2N}$ such that $|\bs|_\infty>a+1$ we have
\begin{equation}\label{P>Q}
\bsE\bigl[\,\eP_{>Q}\zeta(B)\cdot\theta_\bs \eP_{>Q}\zeta(B)\,\bigr]=\sum_{q>Q}\int_B\int_B\bsE\bigl[\,\rho_q(\bt)]\rho_q(\bu+\bs)\,\bigr] d\bt d\bu,
\end{equation}
where we recall from (\ref{rhoq0}) that
\[
\rho_q(\bt)=\sum_{\gamma\in\eI_m,\;|\gamma|=q} \ba(\gamma) H_\gamma\bigl(\,Y(\bt)\,\bigr),\;\;\eI_m:=\bN^m_0\times\bN_0^{\nu(m)},\;\;\nu(m)=\frac{m(m+1)}{2}.
\]
Thus
\[
\bsE\bigl[\,\rho_q(\bt)\rho_q(\bu+\bs)\,\bigr]=\bsE\Bigl[\,\Bigl(\sum_{\gamma\in\eI_m,\;|\gamma|=q} \ba(\gamma) H_\gamma\bigl(\,Y(\bt)\,\bigr)\,\Bigr)\Bigl(\, \sum_{\gamma\in\eI_m,\;|\gamma|=q} \ba(\gamma) H_\gamma\bigl(\,Y(\bs+\bu)\,\bigr)\,\Bigr)\Bigr]
\]
Arcones' inequality \cite[Lemma 1]{Arc} implies that
\begin{equation}\label{P>Q1}
\bsE\bigl[\,\rho_q(\bt)\rho_q(\bu+\bs)\,\bigr]\leq  K^q\psi(\bs+\bu-\bt)^q \sum_{\gamma\in\eI_m,\;|\gamma|=q} |\ba(\gamma)|^2\gamma!.
\end{equation}
The series $\sum_{\gamma\in\eI_m} |\ba(\gamma)|^2\gamma !$ is divergent because the series $\sum_{\gamma\in\eI_m} \ba(\gamma) H_\gamma(Y)$, $Y=(U,B)$, is the Hermite series decomposition of the distribution $\delta_0(\sqrt{d_m}U) |\det A|$.

On the other hand, for $\gamma=(\alpha,\beta)\in\eI_m$ we  have $\ba(\gamma)=d(\alpha) c(\beta)$, where, according to (\ref{dalph}) we have $d(\alpha)=\frac{1}{\alpha!}(2\pi d_m)^{-\frac{m}{2}}  H_\alpha(0)$. Recalling that
\[
H_{2r}(0)= (-1)^r\frac{(2r)!}{2^rr!},\;\;H_{2r+1}(0)=0.
\]
we deduce that
\[
 (2r)!\Bigl|\frac{1}{(2r)!}H_{2r}(0)\Bigr|^2=\frac{1}{2^{2r}}\binom{2r}{r}\leq 1, 
\]
and
\[
d(\alpha)^2\alpha!\leq C=\frac{1}{(2\pi d_m)^{m/2}}.
\]
This allows us to conclude that 
\[
 \sum_{\gamma\in\eI_m,\;|\gamma|=q} |\ba(\gamma)|^2\gamma!\leq (2\pi d_m)^{-m/2} q^m \sum_{\beta\in\bN_0^{\nu(m)},|\beta|\leq q} c(\beta)^2\beta!\leq C q^m\bsE_{\eS_m^{h_m}}\bigl[\,|\det A|^2\,\bigr].
\]
Using this in (\ref{P>Q}) and (\ref{P>Q1}) we deduce
\[
\bsE\bigl[\,\eP_{>Q}\zeta\bigl(B)\cdot\theta_\bs \eP_{>Q}\zeta\bigl(B)\,\bigr]
\]
\[
\leq \underbrace{C\bsE_{\eS_m^{h_m}}\bigl[\,|\det A|^2\,\bigr]}_{=:C'}\; \sum_{q>Q} q^mK^q\int_B \int_B\psi(\bs+\bu-\bt)^q d\bu d\bt
\]
Hence
\[
\bigl|\,V''_{>Q, N}\bigr|\leq C'\Bigl( \sum_{q>Q}  q^mK^q\rho^{q-1}\Bigr)\Bigl( \sum_{\bs\in\eL_{2N};\;|\bs|_\infty>a+1}\int_B\int_B \psi(\bs+\bu-\bt) d\bu d\bt\Bigr),
\]
where we have used  the fact that for $|s|_\infty\geq a+1$, $|\bu|,|\bt|\leq 1$  we have  $\psi(\bs+\bu-\bt)<\rho$. Since $\rho<\frac{1}{K}$, the sum 
\[
\sum_{q>Q}  q^mK^q\rho^{q-1}
\]
is the tail of a convergent power series.   On the other hand,
\[
\sum_{\bs\in\eL_{2N};\;|\bs|_\infty>a+1}\int_B\int_B \psi(\bs+\bu-\bt) d\bu d\bt\leq\sum_{\bs\in\eL_{2n}}\int_{[-1,1]^m}\psi(\bs+\bu)\leq 2\int_{\bR^m}\psi(\bu) d\bu<\infty.
\]
This proves   that $\sup_N|V''_{>Q, N}|$ goes to zero as $Q\to\infty$ and  completes the proof of  (\ref{supN}). The claim (\ref{vinfty}) follows immediately from (\ref{supN}).
\end{proof}

\begin{lemma} The asymptotic variance $V_\infty$ is positive. More precisely,
\begin{equation}\label{v2inf}
V_{2,\infty}>0.
\end{equation}
\end{lemma}

\begin{proof}  From (\ref{vq2}) we deduce
\begin{equation}
V_{2,\infty}=\int_{\bR^m} \bsE\bigl[ \,\rho_2(0)\rho_2(\bu)\,\bigr]d\bu,
\label{v2infty}
\end{equation}
where, according to (\ref{rhoq}) we have
\[
\rho_2(\bt)= \sum_{\substack{\alpha\in \bN_0^m, \,n\in\bN_0\\|\alpha|+n=2}}d(\alpha) H_\alpha\bigl(\, U(\bt)\,\bigr)f_n\bigl(\, A(\bt)\,\bigr).
\]
The second chaos decomposition  $f_2(A)$  is a linear combination of the polynomials $\bar{p}(A)$ and $\bar{q}(A)$ defined in (\ref{barpq}) where $v=h_m$.

In the above  sum  the only nontrivial terms correspond to $\alpha=0$ or  $\alpha=(2\delta_{i1},2\delta_{i2},\dotsc,2\delta_{im})$, $i=1,\dotsc,m$. In each of these latter cases we  deduce  from (\ref{hn0}) that 
\[
d(\alpha):=-\frac{1}{2}d(0),\;\;d(0)\stackrel{(\ref{dalph})}{=}(2\pi d_m)^{-\frac{m}{2}} ,
\]
and we conclude  that
\[
\rho_2(\bt)= d(0)\Bigl(f_2(\,A(\bt)\,) -\frac{f_0(A)}{2}\sum_{i=1}^m H_2(\,U_i(\bt)\,)\Bigr)
\]
\[
=d(0)\Bigl(x\bar{p}(A(\bt))+y\bar{q}(A(\bt)) -\frac{f_0(A)}{2}{\sum_{i=1}^m H_2(\,U_i(\bt)\,)}\Bigr).
\]
For uniformity  we set
\[
z:= -\frac{f_0(A)}{2}
\]
so that
\[
\rho_2(\bt)=d(0)\Bigl(\,  x\bar{p}(A)+y\bar{q}(A) +z\sum_{i=1}^m H_2(U_i)\,\Bigr).
\]
We first  express the polynomials  $\bar{p}(A)$  and $\bar{q}(A)$ in terms of Hermite polynomials. We set
\[
\ha_{ij}:=\begin{cases}
\frac{1}{\sqrt{3h_m}} a_{ii}, & i=j,\\
\frac{1}{\sqrt{h_m}} a_{ij}, & i\neq j.
\end{cases}
\]
We have
\[
(\tr A)^2 =3h_m \left(\sum_i \ha_{ii}\right)^2=3h_m\sum_i \ha_{ii}^2+6h_m\sum_{i<j}\ha_{ii}\ha_{jj}
\]
\[
=3h_m\sum_i\bigl(\,H_2(\ha_{ii})+1\,\bigr) +6h_m\sum_{i<j}H_1(\ha_{ii})H_1(\ha_{jj}).
\]
\[
\bar{p}(A)=(\tr A)^2 -m(m+2)h_m=h_m\left( 3\sum_i H_2(\ha_{ii})+ 6\sum_{i<j}H_1(\ha_{ii})H_1(\ha_{jj})-m(m-1)\,\right).
\]
We have
\[
\tr A^2 = 3h_m\sum_i \ha_{ii}^2 +2h_m\sum_{i<j} \ha_{ij}^2= 3h_m\sum_i H_2(\ha_{ii})+2h_m \sum_{i<j}H_2(\ha_{ij})+m(m+2)h_m,
\]
\[
\bar{q}(A)=h_m \left(\,3\sum_i H_2(\ha_{ii})+2 \sum_{i<j}H_2(\ha_{ij})\,\right).
\]
Define
\[
F_0(\bt)=\sum_iH_2\bigl(\,U_i(\bt)\,\bigr),\;\;F_1(\bt)=\sum_iH_2\bigl(\,\ha_{ii}(\bt)\,\bigr),\;\;F_2(\bt)=\sum_{i<j} H_2\bigl(\,\ha_{ij}(\bt)\,\bigr)
\]
\[
F_3(\bt)=\sum_{i<j} H_1\bigl( \, \ha_{ii}(\bt)\,\bigr) H_1\bigl(\, \ha_{jj}(\bt)\,\bigr).
\]
Thus
\[
\rho_2(\bt)= d(0)\Bigl(xh_m\,{\bigr(\,3F_1(\bt) +6 F_3(\bt)-m(m-1)\,\bigr)}+yh_m\,{\bigl(\, 3F_1(\bt)+2 F_2(\bt)\,\bigr)} +z F_0(\bt)\,\Bigr).
\]
\[
\bsE[F_1(0)]=\bsE[F_2(0)]=0,
\]
\[
\bsE[F_3(\bt)]=\bsE[F_3(0)]=\frac{m(m-1)}{2}\bsE[\ha_{11}(0)\ha_{22}(0)]= \frac{m(m-1)}{6}
\]
We set
\[
\widehat{F}_3(\bt)=F_3(\bt)-\bsE[F_3(\bt)].
\]
Then
\[
\bsE[ F_0(0)\widehat{F}_3(\bt)]=\bsE[ F_0(0)F_3(\bt)],\;\;\bsE[ F_1(0)\widehat{F}_3(\bt)]=\bsE[ F_1(0)F_3(\bt)],
\]
\[
\bsE[ F_2(0)\widehat{F}_3(\bt)]=\bsE[ F_2(0)F_3(\bt)],
\]
\begin{equation}\label{rho2herm}
\rho_2(\bt)=d(0)\Bigl(\, 3xh_m\,\underbrace{\bigr(\,F_1(\bt) +2 \widehat{F}_3(\bt))\,\bigr)}_{=:\eZ_1(\bt)}+yh_m\underbrace{\bigl(\, 3F_1(\bt)+2 F_2(\bt)\,\bigr)}_{=:\eZ_2(\bt)} +z F_0(\bt)\,\Bigr).
\end{equation}
To estimate  $ \bsE\bigl[ \,\rho_2(0)\rho_2(\bu)\,\bigr]$ we will rely on the following consequences of the  Diagram Formula \cite{BrMaj}, \cite[Thm. 3.12]{Jan}.

\begin{lemma}  Suppose that $X_1,X_2,X_3,X_4$ are  centered   Gaussian random variables such that
\[
\bsE[ X_i^2]=1,\;\;\bsE[X_iX_j]=c_{ij},\;\;\forall i, j=1,2,4,\;\;i\neq j.
\]
Then
\begin{subequations}\label{hdiag}
\begin{equation}
\bsE\bigl[\,H_1(X_1)H_1(X_2)\,\bigr]=c_{12},
\end{equation}
\begin{equation}
\bsE\bigl[ \,H_2(X_1) H_2(X_2)\,\bigr]=2c_{12}^2,\;\;\bsE\bigl[\,H_2[X_1]H_1(X_2)H_1(X_3)\,\bigr]=2c_{12}c_{13},
\end{equation}
\begin{equation}
\bsE\bigl[ H_1(X_1)H_1(X_2)H_1(X_3)H_1(X_4)\,\bigr]=c_{12}c_{34}+c_{13}c_{24}+c_{14}c_{23}.
\end{equation}
\end{subequations}\qed
\end{lemma}

To compute  the expectations involved in (\ref{v2infty}) we need to know the covariances between $\ha_{ij}(0)$ and $\ha_{jk}(\bt)$. These are  determined by  the covariance kernel 
\[
C(\bt)=\eF[\mu(\blam) ],
\]
where $\eF$ denotes the Fourier transform.  For any  $i_1,\dots, i_k\in\{1,\dotsc m\}$  we set
\[
C_{i_1\dotsc i_k}(\bt):=\pa^k_{t_{i_1}\dotsc t_{i_k}}C(\bt),\;\;\eF\bigl[\, (-\ii)^{k} \lambda_{i_1}\cdots \lambda_{i_k} \mu(\blam)\,\bigr].
\]
We have
\[
\bsE\bigl[\pa^k_{t_{i_1}\dotsc t_{i_k}}X(0)\pa^\ell_{t_{j_1}\dotsc t_{j_\ell}}X(\bt)\,\bigr]=(-1)^{k}C_{i_1\dotsc i_k j_1\dotsc j_\ell}(\bt)
\]
\[
U_i(\bt) =\frac{1}{\sqrt{d_m}}\pa_{t_i }X(t),\;\;\ha_{ii}(\bt)=\frac{1}{\sqrt{3h_m}} \pa^2_{t_i}X(\bt),\;\;\ha_{ij}(\bt)=\frac{1}{\sqrt{h_m}} \pa^2_{t_it_j} X(t).
\]
Recalling that the spectral measure $\mu(d\blam)$ has the form
\[
\mu(\blam)=w(|\blam|) d\blam,
\]
we introduce the functions
\[
M_{i_1\dotsc i_k}(\blam):=\lambda_{i_1}\cdots \lambda_{i_k} w(|\blam|)
\]
and denote by $\eF_{i_1\dotsc i_k}$ their Fourier transforms
\[
\eF_{i_1\dotsc i_k}(\bt)=(2\pi)^{-m/2}\int_{\bR^m}e^{-\ii(\bt,\blam)}M_{i_1\dotsc i_k}(\blam) d\blam.
\]
Then
\begin{subequations}\label{f}
\begin{equation}\label{a}
\bsE\bigl[ U_i(0)U_j(\bt)\,\bigr]=-\frac{1}{d_m} C_{ij}=\frac{1}{d_m}\eF_{ij}(t),
\end{equation}
\begin{equation}
\bsE\bigl[ U_i(0) \ha_{jj}(\bt)\bigl]=-\bsE\bigl[\ha_{jj}(0)U_i(\bt)]\\
=-\frac{1}{\sqrt{3h_md_m}} C_{ijj}(\bt)=\frac{\ii }{\sqrt{3h_md_m}} \eF_{ijj}(\bt),\;\;\forall i,j,
\end{equation}
\begin{equation}
\begin{split}
\bsE\bigl[ U_i(0) \ha_{jk}(\bt)\bigl]=-\bsE\bigl[\ha_{jk}(0)U_i(\bt)]
=-\frac{1}{\sqrt{h_md_m}} C_{ijk}(\bt)\\
=\frac{\ii}{\sqrt{h_md_m}} \eF_{ijk}(\bt)\;\;\forall i,j,k,\;\;j\neq k,
\end{split}
\end{equation}
\begin{equation}
\bsE\bigl[ \ha_{ii}(0)\ha_{jj}(\bt)\bigr]=\frac{1}{3h_m}C_{iijj}(\bt)=\frac{1}{3h_m}\eF_{iijj}(\bt),
\end{equation}
\begin{equation}
\begin{split}
\bsE\bigl[ \ha_{ii}(0)\ha_{jk}(\bt)\bigr]= \bsE\bigl[ \ha_{jk}(0)\ha_{ii}(\bt)\bigr]=\frac{1}{h_m\sqrt{3}}C_{iijk}(\bt)\\
=\frac{1}{h_m\sqrt{3}}\eF_{iijk}(\bt),\;\;\forall i,j,k,\;\;j\neq k,
\end{split}
\end{equation}
\begin{equation}
\begin{split}\label{ff}
\bsE\bigl[ \ha_{ij}(0)\ha_{k\ell}(\bt)\bigr]= \bsE\bigl[ \ha_{k\ell}(0)\ha_{ij}(\bt)\bigr]
=\frac{1}{h_m}C_{iijk}(\bt)\\
=\frac{1}{h_m}\eF_{ijk\ell}(\bt),\;\;\forall i,j,k,\ell,\;\;i\neq j, \;\;k\neq \ell.
\end{split}
\end{equation}
\end{subequations}
We have
\[
\bsE[F_0(0) F_0(\bt)]=\sum_{i,j}\bsE[H_2(U_i(0)H_2(U_j(\bt))]=\frac{2}{d_m^2}\sum_{i,j}\eF_{ij}(\bt)^2.
\]

Using the fact that the Fourier transform is an isometry  and the equality $M_{ij}^2=M_{ii}M_{jj}$ we deduce
\[
\int_{\bR^m}\eF_{ij}(\bt)^2 d\bt=\int_{\bR^m}M_{ij}(\blam)^2 d\blam=\int_{\bR^m}M_{ii}(\blam)M_{jj}(\blam) d\blam
\]
and thus,
\begin{equation}\label{z0z0}
\int_{\bR^m}\bsE[F_0(0) F_0(\bt)] d\bt=\left\langle \frac{\sqrt{2}}{d_m}\sum_{i} M_{ii}, \frac{\sqrt{2}}{d_m}  \sum_{j} M_{jj}\right\rangle_{L^2}
\end{equation}
\[
\bsE\bigl[F_0(0) F_1(\bt)] =\sum_{i,j}\bsE[ H_2(U_i(0))H_2(\ha_{jj}(t))] =\frac{2}{3h_md_m}\sum_{i,j}\bigl(\ii\eF_{ijj}(\bt)\,)^2.
\]
Since $M_{ijj}^2=M_{ii}M_{jjjj}$ we deduce
\begin{equation}\label{z0f1}
\int_{\bR^m}\bsE\bigl[F_0(0) F_1(\bt)]  d\bt= \frac{2}{3h_md_m} \sum_{i,j}\Vert M_{ijj}\Vert^2_{L^2}=\left\langle \frac{\sqrt{2}}{d_m}\sum_i M_{ii},\frac{\sqrt{2}}{3h_m}\sum_j M_{jjjj}\right\rangle_{L^2}.
\end{equation}
Arguing similary we deduce
\[
\int_{\bR^m}\bsE\bigl[F_0(0) F_1(\bt)]  d\bt=\int_{\bR^m}\bsE\bigl[F_0(\bt) F_1(0)]  d\bt.
\]
\[
\bsE[\,F_0(0) F_2(\bt)]=\sum_{i, j<k} H_2(U_i(0)) H_2(\ha_{jk}(\bt))=2\sum_{i,j<k} \frac{\ii}{\sqrt{d_mh_m}}\eF_{ijk}(\bt),
\]
\begin{equation}\label{z0f2}
\int_{\bR^m}\bsE[\,F_0(0) F_2(\bt)]d\bt=\left\langle\frac{\sqrt{2}}{d_m} \sum_i M_{ii},\frac{\sqrt{2}}{h_m}\sum_{j<k} M_{jjkk}\,\right\rangle_{L^2}=\int_{\bR^m}\bsE[\,F_0(\bt) F_2(0)]d\bt.
\end{equation}
\[
\bsE[F_0(0)F_3(\bt)]=\sum_{i,j<k} \bsE\bigl[ H_2(U_i(0)) H_1(\ha_{jj}(\bt) H_1(\ha_{kk}(\bt))\,\bigr]=-2\sum_{i,j<k}\frac{1}{3d_mh_m}\eF_{ijj}(\bt)(\bt)\eF_{ikk}(\bt),
\]
\begin{equation}\label{z0f3}
\begin{split}
\int_{\bR^m}\bsE[F_0(0)F_3(\bt)]d\bt =\frac{2}{3h_md_m}\sum_{i,j<k} \left\langle M_{ijj}, M_{ikk}\right\rangle_{L^2}\\ =\left\langle\frac{\sqrt{2}}{d_m}\sum_i M_{ii},\frac{\sqrt{2}}{3h_m}\sum_{j<k} M_{jjkk}\right\rangle_{L^2} =\int_{\bR^m}\bsE[F_0(\bt)F_3(0)]d\bt.
\end{split}
\end{equation}
We have
\[
\bsE\bigl[\, F_1(0) F_1(\bt)\,\bigr]=\sum_{i,j} \bsE[\, H_2(\ha_{ii}(0))H_2(\ha_{jj}(t))\,\bigr]=\frac{2}{9h_m^2}\sum_{i,j}\eF_{iijj}(t)^2
\]
\begin{equation}\label{f1f1}
\int_{\bR^m} \bsE\bigl[\, F_1(0) F_1(\bt)\,\bigr] d\bt=\left\langle\frac{\sqrt{2}}{3h_m}\sum_i M_{iiii}, \frac{\sqrt{2}}{3h_m}\sum_i M_{iiii}\right\rangle.
\end{equation}
\[
\bsE\bigl[\, F_1(0) F_2(\bt)\,\bigr]=\sum_{i,j<k} \bsE[\, H_2(\ha_{ii}(0))H_2(\ha_{jk}(t))\,\bigr]=\frac{2}{3h_m^2}\sum_{i,j<k} \eF_{iijk}(\bt)^2,
\]
\begin{equation}\label{f1f2}
\int_{\bR^m}\bsE\bigl[\, F_1(0) F_2(\bt)\,\bigr]d\bt=\left\langle\frac{\sqrt{2}}{3h_m} \sum_{i}M_{iiii},\frac{\sqrt{2}}{h_m}\sum_{j<k}M_{jkjk}\right\rangle_{L^2}.
\end{equation}
\[
\bsE\bigl[\, F_1(0) F_3(\bt)\,\bigr]=\sum_{i,j<k}\bsE[\, H_2(\ha_{ii}(0)) H_1(\ha_{jj}(\bt)) H_1(\ha_{kk}(\bt)\,\bigr]=\frac{2}{9h_m^2}\sum_{i,j<k} \eF_{iijj}(t)\eF_{iikk}(\bt).
\]
\begin{equation}\label{f1f3}
\int_{\bR^m} \bsE\bigl[\, F_1(0) F_3(\bt)\,\bigr] d\bt =\left\langle \frac{\sqrt{2}}{3h_m} \sum_{i}M_{iiii}, \frac{\sqrt{2}}{3h_m}\sum_{j<k} M_{jjkk}\right\rangle_{L^2}.
\end{equation}
\[
\bsE\bigl[\, F_2(0) F_2(\bt)\,\bigr]=\sum_{i<j, k<\ell} \bsE\bigl[\,H_2(\ha_{ij}(0))H_2(\ha_{k\ell}(\bt)\,\bigr]=2\sum_{i<j,k<\ell}\eF_{ijk\ell}(t)^2,
\]
\begin{equation}\label{f2f2}
\int_{\bR^m} \bsE\bigl[\, F_2(0) F_2(\bt)\,\bigr] d\bt=\left\langle\frac{\sqrt{2}}{h_m}\sum_{i<j}M_{ijij}, \frac{\sqrt{2}}{h_m}\sum_{i<j}M_{ijij}\right\rangle_{L^2}.
\end{equation}
\[
 \bsE\bigl[\, F_2(0) F_3(\bt)\,\bigr]=\sum_{i<j, k<\ell} \bsE\bigl[\,H_2(\ha_{ij}(0)) H_1(\ha_{kk}(\bt))H_1(\ha_{\ell\ell}(\bt))\bigr] =\frac{2}{3h_m^2} \sum_{i<j, k<\ell} \eF_{ijkk}(\bt)\eF_{ij\ell\ell}(\bt),
\]
\begin{equation}\label{f2f3}
\int_{\bR^m} \bsE\bigl[\, F_2(0) F_3(\bt)\,\bigr]=\left\langle\frac{\sqrt{2}}{h_m}\sum_{i<j}M_{ijij},\frac{\sqrt{2}}{3h_m}\sum_{k<\ell}M_{kk \ell \ell}\right\rangle_{L^2}.
\end{equation}
\[ 
 \bsE\bigl[\, F_3(0) F_3(\bt)\,\bigr]=\sum_{i<j, k<\ell} \bsE\bigl[\,H_1(\ha_{ii}(0))H_1(\ha_{jj}(0)) H_1(\ha_{kk}(\bt))H_1(\ha_{\ell\ell}(\bt))\bigr]
 \]
\[
=\sum_{i<j, k<\ell}\bsE[\ha_{ii}(0)\ha_{jj}(0)]\bsE[\ha_{kk}(\bt)\ha_{\ell\ell}(\bt)]+\frac{1}{9h_m^2}\sum_{i<j, k<\ell}\Bigl(\eF_{iikk}(\bt)\eF_{jj\ell\ell}(\bt)+\eF_{ii\ell\ell}(\bt)\eF_{jjkk}(\bt)\Bigr)
\]
\[
=\bsE[F_3(0)]^2 + \frac{1}{9h_m^2}\sum_{i<j, k<\ell}\Bigl(\eF_{iikk}(\bt)\eF_{jj\ell\ell}(\bt)+\eF_{ii\ell\ell}(\bt)\eF_{jjkk}(\bt)\Bigr).
\]
\[
\bsE\bigl[\, \whF_3(0) \whF_3(\bt)\,\bigr]= \bsE\bigl[\, F_3(0) F_3(\bt)\,\bigr]- \bsE[F_3(0)]^2=\frac{1}{9h_m^2}\sum_{i<j, k<\ell}\Bigl(\eF_{iikk}(\bt)\eF_{jj\ell\ell}(\bt)+\eF_{ii\ell\ell}(\bt)\eF_{jjkk}(\bt)\Bigr).
\]
We deduce
\[
\int_{\bR^m}\bsE\bigl[\, \whF_3(0) \whF_3(\bt)\,\bigr] d\bt= \frac{1}{9h_m^2}\sum_{i<j, k<\ell}\int_{\bR^m}\Bigl(\eF_{iikk}(\bt)\eF_{jj\ell\ell}(\bt)+\eF_{ii\ell\ell}(\bt)\eF_{jjkk}(\bt)\Bigr) d\bt
\]
 and we conclude
 \begin{equation}\label{f3f3}
 \int_{\bR^m}\bsE\bigl[\, \whF_3(0) \whF_3(\bt)\,\bigr] d\bt=\left\langle\frac{\sqrt{2}}{3h_m}\sum_{i<j}M_{ijij},\frac{\sqrt{2}}{3h_m}\sum_{k<\ell}M_{kk \ell \ell}\right\rangle_{L^2}
\end{equation}

To put the above equalities in perspective  we introduce the functions
\[
G_0=\frac{\sqrt{2}}{d_m}\sum_i M_{ii},\;\;G_1=\frac{\sqrt{2}}{3h_m}\sum_i M_{iiii},\;\;G_2=\frac{\sqrt{2}}{h_m}\sum_{j<k} M_{jjkk} ,\;\; G_3=\frac{1}{3}G_2.
\]
 The assumption {\bf A3} implies that $G_0,G_1,G_2\in L^2(\bR^m,d\blam)$.  Using  the notation
 \[
 F_i\bullet F_j:=\int_{\bR^m}\bsE\bigl[\,F_i(0), F_j(\bt)\,\bigr].
 \]
 we can rewrite the  equalities (\ref{z0z0}, ..., \ref{f3f3}) in a more concise form
 \[
 F_i\bullet F_j=F_j\bullet F_i=\lan G_i, G_j\ran_{L^2},\;\;F_i\bullet \whF_3=F_i\bullet F_3,\;\;\forall i,j=0,1,2,
 \]
 \[
\whF_3\bullet F_i= F_i\bullet \whF_3 =\lan G_i, G_3\rangle_{L^2},\;\;\forall i=0,\dotsc,3.
 \]
 From (\ref{v2infty}) and (\ref{rho2herm}) we deduce
 \[
V_{2,\infty}= \int_{\bR^m}\rho_2(\bt) d\bt=d(0)^2\Bigl( 3xh_m \eZ_1+yh_m \eZ_2+z F_0)\,\Bigr)\bullet \Bigl( 3xh_m \eZ_1+yh_m \eZ_2+z F_0)\,\Bigr)
 \]
 \[
 =d(0)^2\Bigl\Vert 3xh_m(G_1+2G_3) +yh_M(3G_1+2G_2) +zG_0\Bigr\Vert_{L^2}^2
 \]
 \[
 =(d(0))^2\Bigl\Vert 3xh_m\Bigl(G_1+\frac{2}{3}G_3\Bigr) +3yh_M\Bigl(G_1+\frac{2}{3}G_2\Bigr) +zG_0\Bigr\Vert_{L^2}^2
\]
\[
=d(0)^2\Bigl\Vert 3h_m(x+y) \Bigl(G_1+\frac{2}{3}G_2\Bigr) + zG_0\Bigr\Vert^2_{L^2}.
\]
The functions $G_1+\frac{2}{3}G_2$  and $G_0$ are linearly independent and 
\[
z=-\frac{1}{2}f_0(A)=-\frac{1}{2}\bsE\bigl[\,|\det A|\,\bigr]\neq 0.
\]
 Hence $V_{2,\infty}>0$. \end{proof}

 This concludes the proof  of Proposition  \ref{prop: el21}.  \end{proof}
 
\begin{remark} The numbers $x,y$  that describe $f_2(A)$, the  2nd chaos component of $f(A)$ seem  hard to compute in general. In Appendix \ref{s: a}   we describe their  large $m$ asymptotics; see (\ref{xym}). \qed
 \end{remark}
 
 \subsection{Conclusion}\label{ss: 34} As  explained in  Subsection \ref{ss: 23},   to conclude the proof of Theorem \ref{th: main} it suffices to establish  the asymptotic normality of the sequences
 \[
\zeta_{N, q}=\frac{1}{(2N)^{m/2}} \int_{C_N} \rho_q(\bt) d\bt,\;\;\forall q\geq 2.
 \]
This follows from the fourth-moment theorem \cite[Thm. 5.2.7]{NP}, \cite{NuaPe}.      Here are the details. 

Recall  from \cite[IV.1]{Jan} that we have a surjective  isometry $\Theta_q: \eX^{\odot q}\to\eX^{:q:}$ , where  $\eX^{\odot q}$ is the $q$-th symmetric power  and $\eX^{:q:}$  is the $q$-th  chaos component of $\widehat{\eX}$.  The multiple  Ito integral $\bsI_q$  is then the map
\[
\bsI_q=\frac{1}{\sqrt{q!}} \Theta_q.
\]
We can write $\zeta_{N,q}$  as a multiple  Ito integral
\[
\zeta_{N, q} = \bsI_q\bigl[\,g_q^N\,\bigr], \;\;g_q^N\in\eX^{\odot n}.
\]
According to  \cite[Thm.5.2.7(v)]{NP}, to prove that $\zeta_{N,q}$ converge in law to  a normal variable it suffices to show that
\begin{equation}\label{EL18}
\lim_{N\to\infty}\Vert g_q^N\otimes_r g_q^N\Vert_{\eX^{\otimes(2q-2r)}}=0,\;\;\forall r=1,\dotsc, q-1.
\end{equation}
 To show this we invoke   arguments  following  the inequality (18) in the proof of \cite[Prop. 2.4]{EL}  which extend without  any  modifications  to the setup in our paper. 
\appendix

\section{Asymptotics of some Gaussian integrals}
\label{s: a}

We  want to give an approximate description of the   2nd chaos component of $f(A)=|\det A|$ when $m\gg 0$.

Observe that if $u:\eS_m\to \bR$ is a  continuous function, homogeneous of degree $k$,  then for any $v>0$  we have
 \[
 \bsE_{\eS_m^v}[u(A)]=(2v)^{\frac{k}{2}} \bsE_{\eS_m^{1/2}}[u(A)].
  \]
 \begin{proposition}\label{prop: dm} Set $\bsC_m:= 2^{\frac{3}{2}}\Gamma\left(\frac{m+3}{2}\right)$. We have the following asymptotic estimates as $m\to\infty$
 \begin{equation}\label{fainf}
 \bsE_{\eS_m^{1/2}}\bigl[\,f(A)\,\bigr]\sim\bsC_m\sqrt{\frac{2}{\pi}}\; m^{-\frac{1}{2}}.
 \end{equation}
 \begin{equation}\label{pfainf}
  \bsE_{\eS_m^{1/2}}\bigl[\,p(A)f(A)\,\bigr]\sim \frac{2\bsC_m}{\sqrt{\pi}}\; m^{\frac{3}{2}}.
  \end{equation}
  \begin{equation}\label{qfainf}
  \bsE_S\bigl[\, q(A) f(A)\,\bigr] \sim\frac{\bsC_m}{\sqrt{2\pi}}\; m^{\frac{7}{2}}.
  \end{equation}
 \end{proposition}

\begin{proof} We need to make  a brief detour in the world of random matrices.

We have a   \emph{Weyl integration formula} \cite {AGZ} which states that if  $f: \eS_m\ra \bR$ is a measurable  function which  is invariant under  conjugation, then the     value $f(A)$ at $A\in\eS_m$ depends only on the eigenvalues $\lambda_1(A)\leq \cdots \leq \lambda_n(A)$ of $A$ and we  have
\begin{equation}
\bsE_{\GOE_m^v}\bigl(\,f(X)\,\bigr)=\frac{1}{\bsZ_m(v)} \int_{\bR^m}  f(\lambda_1,\dotsc ,\lambda_m) \underbrace{\left(\prod_{1\leq i< j\leq m}|\lambda_i-\lambda_j| \right)\prod_{i=1}^m e^{-\frac{\lambda_i^2}{4v}} }_{=:Q_{m,v}(\lambda)} |d\lambda_1\cdots d\lambda_m|,
\label{eq: weyl}
\end{equation}
where $\bsZ_m(v)$  can be computed via Selberg integrals, \cite[Eq. (2.5.11)]{AGZ}, and we have
\begin{equation}
\bsZ_m(v)=(2v)^{\frac{m(m+1}{4}}\bsZ_m,\;\;\bsZ_m=(2\pi)^{\frac{m}{2}} m!\prod_{j=1}^{m}\frac{\Gamma(\frac{j}{2})}{\Gamma(\frac{1}{2})}=2^{\frac{m}{2}}m!\prod_{j=1}^m \Gamma\left(\frac{j}{2}\right).
\label{eq: zm}
\end{equation}

For any positive integer $n$ we define the \emph{normalized} $1$-point correlation function $\rho_{n,v}(x)$ of $\GOE_n^v$ to be
\[
\rho_{n,v}(x)= \frac{1}{\bsZ_n(v)}\int_{\bR^{n-1}} Q_{n,v}(x,\lambda_2,\dotsc, \lambda_n) d\lambda_1\cdots d\lambda_n.
\]
For any Borel measurable function $f:\bR\ra \bR$  we have \cite[\S 4.4]{DG}
 \begin{equation}
 \frac{1}{n}\bsE_{\GOE_n^v} \bigl(\,\tr f(X)\,\bigr)= \int_{\bR} f(\lambda) \rho_{n,v}(\lambda) d\lambda. 
 \label{eq: 1pcor}
 \end{equation}
The equality (\ref{eq: 1pcor}) characterizes $\rho_{n,v}$. Let us  observe that for any constant $c>0$, if 
\[
A\in \GOE_n^v\Llra cA\in \GOE_n^{c^2v}.
\]
Hence for any Borel set $B\subset \bR$   we have
\[
\int_{cB} \rho_{n,c^2v}(x) dx =\int_B \rho_{n,v} (y) dy,
\]
and we conclude that
 \begin{equation}
 c\rho_{n,c^2v}(cy)= \rho_{n,v}(y),\;\;\forall n,c,y.
 \label{eq: resc-cor}
 \end{equation}
 The behavior of $\rho_{n,v}$ for large $n$ is described the the celebrated  \emph{Wigner semicircle theorem}.
 \begin{theorem}[Wigner]  As $n\to \infty$, the probability measures
 \[
 \rho_{n,vn^{-1}}(\lambda)|d\lambda| =n^{1/2}\rho_{n,v}\bigl(\, n^{1/2}\lambda\,\bigr)|d\lambda|
 \]
 converge weakly to  the semicircle distribution
 \[
 \rho_{\infty, v}(\lambda)|d\lambda|=\bsI_{|\lambda|\leq 2\sqrt{v}}\frac{1}{2\pi v}\sqrt{4v-\lambda^2}|d\lambda|.\proofend
 \]
 \end{theorem}
  
 We have the following result of Y. Fyodorov \cite{Fy}; see also  \cite[Lemmas C.1, C2.]{Ni2014}.
\begin{lemma}  Suppose $v>0$. Then for any $\lambda\in\bR$ we have
\begin{equation}\label{fy1}
\bsE_{\GOE_m^v} \bigl(\,|\det(\lambda +B|\,\bigr)=(2v)^{\frac{m+1}{2}}\bsC_m e^{\frac{c^2}{4v}}\rho_{m+1,v}(\lambda),\;\;\bsC_m:= 2^{\frac{3}{2}}\Gamma\left(\frac{m+3}{2}\right) .
\end{equation}
\begin{equation}\label{fy2}
\begin{split}
\bsE_{\eS_m^{v}}\bigl(\,|\det (A)|\,\bigr)=(2v)^{\frac{m+1}{2}}\frac{\bsC_m}{\sqrt{2\pi v}}\int_\bR\bsE_{\GOE_m^v} \bigl(\,|\det(\lambda +B|\,\bigr)e^{-\frac{\lambda^2}{2v}} d\lambda\\
= (2v)^{\frac{m+1}{2}} \frac{\bsC_m}{\sqrt{2\pi v}}\int_{\bR}\rho_{m+1,v}(\lambda) e^{-\frac{x^2}{4v}} d\lambda.
\end{split}
\end{equation}\qed
\label{lemma: exp-det}
\end{lemma}

We will also need the following asymptotic estimates.

\begin{lemma} Let $k$ be a nonnegative integer. Then
\begin{equation}\label{asyrho} 
\frac{1}{\sqrt{2\pi}}\int_\bR\rho_{n,\frac{1}{2}}(\lambda)\lambda^{2k} e^{-\frac{\lambda^2}{2v}}d\lambda\sim\frac{v^k\sqrt{2v}(2k-1)!!}{\sqrt{\pi n}}\;\;\mbox{as $n\to\infty$}
\end{equation}
\end{lemma}
\begin{proof} Consider the function
\[
w(\lambda)=\frac{\lambda^{2k}}{v^{\frac{2k+1}{2}}(2k-1)!!}\frac{e^{-\frac{\lambda^2}{2v}}}{\sqrt{2\pi}}.
\]
Then
\[
\int_\bR w(\lambda) d\lambda =1.
\]
We set $w_n(\lambda):=\sqrt{n}w(\sqrt{n}\lambda)$. The probability measures $w_n(\lambda)d\lambda$ converge to $\delta_0$ and we have
\[
\frac{1}{\sqrt{2\pi}}\int_\bR\rho_{n,\frac{1}{2}}(\lambda)\lambda^{2k} e^{-\frac{\lambda^2}{2v}}d\lambda=\frac{v^{\frac{2k+1}{2}}(2k-1)!!}{\sqrt{n}} \int_\bR\rho_{n,\frac{1}{2n}}(\lambda)w_n(\lambda) d\lambda.
\]
Arguing exactly as in \cite[Sec. 4.6]{Ni2015} we deduce
\[
\lim_{n\to\infty}\int_\bR\rho_{n,\frac{1}{2n}}(\lambda)w_n(\lambda) d\lambda = \rho_{\infty,\frac{1}{2}}(0)=\sqrt{\frac{2}{\pi}}.
\]
\end{proof}

The  estimate  (\ref{fainf})  follows from (\ref{fy2}) and (\ref{asyrho}).

 To simplify the notation we set
\[
\bsE_G:=\bsE_{\GOE_m^{1/2}},\;\;\bsE_S:=\bsE_{\eS_m^{1/2}}.
\]
Let us observe that the equality (\ref{eq: smgoe}) implies that
\begin{subequations}
\begin{equation}\label{sp}
\bsE_S\bigl[\,p(A)f(A)\,\bigr]=\frac{1}{\sqrt{\pi}}\int_\bR\bsE_G\bigl[\, p(\lambda+B)f(\lambda+B)\,\bigr] e^{-\lambda^2} d\lambda,
\end{equation}
\begin{equation}\label{sq}
\bsE_S\bigl[\,q(A)f(A)\,\bigr]=\frac{1}{\sqrt{\pi}}\int_\bR\bsE_G\bigl[\, q(\lambda+B)f(\lambda+B)\,\bigr] e^{-\lambda^2} d\lambda.
\end{equation}
\end{subequations}

To estimate  $\bsE_S\bigl[\,p(A)f(A)\,\bigr]$ and $\bsE_S\bigl[\,q(A)f(A)\,\bigr]$ for  $m$ large we use a nice trick we learned from Yan Fyodorov. Introduce the functions
\[
\Phi,\Psi:\bR\times (-\infty,1)\to\bR,
\]
\[
\Phi(\lambda, z):=\bsE_G\bigl[\, |\det(\lambda+B)|e^{z\tr(\lambda+B)}\,\bigr],\;\;\Psi(\lambda,z):=\bsE_G\bigl[\, |\det(\lambda+B)|e^{\frac{z}{2}\tr(\lambda+B)^2}\,\bigr].
\]
Obviously
\[
\Phi(\lambda,0)=\Psi(\lambda,0)=\bsE_G\bigl[\,|\det(\lambda+B)|\,\bigr]
\]
so both $\Psi(\lambda,0)$ and $\Psi(\lambda,0)$ can be determined using (\ref{fy1}).  

Observe next that
\begin{subequations}
\begin{equation}\label{fy3}
\Phi''_{zz}(\lambda,0)=\bsE_G\bigl[ \,\bigl(\,\tr(\lambda+B)\,\bigr)^2\,|\det(\lambda+B)|\,\bigr]=\bsE_G\bigl[\,p(\lambda+B)f(\lambda+B)\,\bigr],
\end{equation}
\begin{equation}\label{fy4}
2\Psi''_{zz}(\lambda,0)=\bsE_G\bigl[ \,(\tr (\lambda+B)^2\,|\det(\lambda+B)|\,\bigr]=\bsE_G\bigl[\,q(\lambda+B)f(\lambda+B)\,\bigr].
\end{equation}
\end{subequations}
We have the following key observation.
\begin{lemma}[Y. Fyodorov]\label{lemma: fy}
\begin{subequations}
\begin{equation}\label{fy5}
\Phi(\lambda, z)= e^{m(\frac{z^2}{2}+z\lambda)}\Phi(\lambda+z,0)=e^{-\frac{m\lambda^2}{2}}e^{\frac{m}{2}(\lambda+z)^2}\Phi(\lambda+z,0),
\end{equation}
\begin{equation}\label{fy6}
\Psi(\lambda,z)=\frac{ e^{\frac{m\lambda^2z}{2(1-z)}} }{(1-z)^{\frac{m(m+3)}{4}}} \Psi\left(\,\frac{\lambda }{\sqrt{1-z}},0\,\right).
\end{equation}
\end{subequations}
\end{lemma}

\begin{proof} Using (\ref{eq: gov}) we deduce
\[
\Phi(\lambda, z)=\bsK_m \int_{\eS_m} |\det (\lambda+B)|  e^{z\tr(\lambda+B)-\frac{1}{2}\tr B^2} dB= e^{m(\frac{z^2}{2}+\lambda z)}\bsK_m  \int_{\eS_m} e^{-\frac{1}{2}\tr(B-z)^2} dB
\]
(make the change in variables $C:=B-z$)
\[
= e^{m(\frac{z^2}{2}+\lambda z)}\bsK_m  \int_{\eS_m} |\det (\lambda+z+B)|  e^{-\frac{1}{2}\tr C^2} dB
\]
\[
= e^{m(\frac{z^2}{2}+\lambda z)}\bsE_G\bigl[\,|\det(\lambda+z +C)|\,\bigr]= e^{m(\frac{z^2}{2}+\lambda z)}\Phi(\lambda+z,0).
\]
Similarly, for $z<1$ we have
\[
\Psi(\lambda,z)=\bsK_m\int_{\eS_m}  |\det (\lambda+B)| e^{\frac{z}{2}\tr(\lambda+B)^2-\frac{1}{2}\tr B^2} dB
\]
\[
=e^{\frac{mz\lambda^2}{2}} \int_{\eS_m}  |\det (\lambda+B)| e^{\frac{\lambda z}{2}\tr B-\frac{1}{z}{2}\tr B^2} dB.
\]
Making the change in variables $B=(1-z)^{-1/2}C$ so that
\[
dB=(1-z)^{-\frac{m(m+1)}{4}} dC\;\;\det(\lambda +B)=(1-v)^{-\frac{m}{2}} \det(\lambda\sqrt{1-z}+C).
\]
We deduce
\[
\Psi(\lambda,z)=\frac{e^{\frac{mz\lambda^2}{2}}}{(1-z)^{\frac{m(m+3)}{4}}}\bsK_m\int_{\eS_m} |\det(\lambda\sqrt{1-z}+C)| e^{-\frac{1}{2}\tr(C-\frac{\lambda z}{\sqrt{1-z}})^2} dC
\]
($C-\frac{\lambda z}{\sqrt{1-z}}\to B$)
\[
= \frac{e^{\frac{mz\lambda^2}{2}}}{(1-z)^{\frac{m(m+3)}{4}}}\bsK_m\int_{\eS_m} |\det(\lambda\sqrt{1-z}+\frac{\lambda z}{\sqrt{1-z}}+B)| e^{-\frac{1}{2}\tr B^2} dB
\]
\[
= \frac{e^{\frac{mz\lambda^2}{2}}}{(1-z)^{\frac{m(m+3)}{4}}}\Psi\left(\,\frac{\lambda}{\sqrt{1-z}},0\right).
\]
\end{proof}

\medskip

\noindent {\bf The  asymptotic behavior of} $\bsE_S\bigl[\,p(A)f(A)\,\bigr]$.  Using (\ref{fy5}) we deduce
\[
\Phi''_{zz}(\lambda, 0)=e^{-\frac{m\lambda^2}{2}}\pa^2_{zz}\bigl|_{z=0}\left(\, e^{\frac{m}{2}(\lambda+z)^2}\Phi(\lambda+z,0)\,\right)=e^{-\frac{m\lambda^2}{2}}\frac{d^2}{d\lambda^2}\bigl(\, e^{\frac{m\lambda^2}{2}}\Phi(\lambda,0)\,\bigr).
\]
Using (\ref{sp}) and (\ref{fy3})  we deduce
\[
\bsE_S\bigl[\,p(A)f(A)\,\bigr]=\frac{1}{\sqrt{\pi}}\int_\bR e^{-\frac{m\lambda^2}{2}}\frac{d^2}{d\lambda^2}\bigl(\, e^{\frac{m\lambda^2}{2}}\Phi(\lambda,0)\,\bigr) e^{-\lambda^2} d\lambda
\]
\[
= \frac{1}{\sqrt{\pi}}\int_\bR \frac{d^2}{d\lambda^2}\bigl(\, e^{\frac{m\lambda^2}{2}}\Phi(\lambda,0)\,\bigr) e^{-\frac{m+2}{2}\lambda^2} d\lambda.
\]
Since $\Phi(\lambda,0)$ has polynomial growth  in  $\lambda$, we can integrate by parts in the above equality  and we deduce
\[
\bsE_S\bigl[\,p(A)f(A)\,\bigr]=  \frac{1}{\sqrt{\pi}}\int_\bR e^{\frac{m\lambda^2}{2}}\Phi(\lambda,0) \frac{d^2}{d\lambda^2}\bigl(\,  e^{-\frac{m+2}{2}\lambda^2}\,\bigr) d\lambda
\]
\[
\stackrel{(\ref{fy1})}{=} \frac{\bsC_m}{\sqrt{\pi}}\int_\bR e^{\frac{m\lambda^2}{2}} e^{\frac{\lambda^2}{2}}\rho_{m+1,\frac{1}{2}}(\lambda)\frac{d^2}{d\lambda^2}\bigl(\,  e^{-\frac{m+2}{2}\lambda^2}\,\bigr) d\lambda 
\]
\[
= \frac{\bsC_m}{\sqrt{\pi}}\int_\bR e^{\frac{(m+1)\lambda^2}{2}} \rho_{m+1,\frac{1}{2}}(\lambda)\frac{d^2}{d\lambda^2}\bigl(\,  e^{-\frac{m+2}{2}\lambda^2}\,\bigr) d\lambda
\]
\[
=\frac{\bsC_m(m+2)\sqrt{2}}{\sqrt{2\pi}}\int_\bR\rho_{m+1,\frac{1}{2}}(\lambda)\bigl(\,(m+2)\lambda^2-1\,\bigr) e^{-\frac{\lambda^2}{2}} d\lambda\stackrel{(\ref{asyrho})}{\sim} \frac{\bsC_m(m+2)\sqrt{2}}{\sqrt{m+1}}\bigl(\, m+1\,\bigr)\rho_{\infty,\frac{1}{2}}(0).
\]
We  have thus proved that 
\begin{equation}\label{asy1}
 \bsE_S\bigl[\,p(A)f(A)\,\bigr]\sim \sqrt{\frac{2}{m+1}}\bsC_{m}(m+2)(m+1)\rho_{\infty,\frac{1}{2}}(0)\;\;\mbox{as $m\to\infty$}. 
 \end{equation}
 This proves (\ref{pfainf}).

 \medskip
 
 \noindent {\bf The asymptotic behavior of} $\bsE_S\bigl[\,q(A)f(A)\,\bigr]$.   We set
 \[
 u(z):=\frac{ e^{\frac{m\lambda^2z}{2(1-z)}} }{(1-z)^{\frac{m(m+3)}{4}}} =e^{-\frac{m\lambda^2}{2}}\frac{ e^{\frac{m\lambda^2}{2(1-z)}} }{(1-z)^{\frac{m(m+3)}{4}}}.
 \]
 Then
 \[
 \Psi(\lambda,z)=u(z)\Psi\bigl(\lambda(1-z)^{-1/2},0\bigr),
 \]
 \[
 \Psi_z'(\lambda,z)=u'(z)\Psi\bigl(\lambda(1-z)^{-1/2},0\bigr)+\frac{\lambda}{2}u(z)(1-z)^{-3/2} \Psi'_\lambda\bigl(\lambda(1-z)^{-1/2},0\bigr).
 \]
 \[
 \Psi''_{zz}(\lambda,z)=u''(z)\Psi\bigl(\lambda(1-z)^{-1/2},0\bigr)+ \frac{\lambda}{2}u'(z)(1-z)^{-3/2} \Psi'_\lambda\bigl(\lambda(1-z)^{-1/2},0\bigr)
 \]
 \[
 +\frac{\lambda}{2}\frac{d}{dx}\bigl(\, u(z)(1-z)^{-3/2} \,\bigr) \Psi'_\lambda\bigl(\lambda(1-z)^{-1/2},0\bigr)+\frac{3\lambda^2}{4}u(z)(1-z)^{-4}\Psi''_{\lambda\lambda}\bigl(\lambda(1-z)^{-1/2},0\bigr).
 \]
 Thus
 \[
 \Psi''_{zz}(\lambda,0)=u''(0)\Psi\bigl(\lambda,0\bigr)+ \frac{\lambda}{2}u'(0) \Psi'_\lambda\bigl(\lambda,0\bigr)
 \]
 \[
 +\frac{\lambda}{2}\bigl(\, u'(0)+\frac{3}{2}u(0) \,\bigr) \Psi'_\lambda\bigl(\lambda,0\bigr)+\frac{3\lambda^2}{4}u(0)\Psi''_{\lambda\lambda}\bigl(\lambda,0\bigr)
 \]
\[
=u''(0)\Psi\bigl(\lambda,0\bigr)+\frac{\lambda}{2}\Bigl(\, 2u'(0)+\frac{3}{2}u(0) \,\Bigr) \Psi'_\lambda\bigl(\lambda,0\bigr)+\frac{3\lambda^2}{4}u(0)\Psi''_{\lambda\lambda}\bigl(\lambda,0\bigr).
\]
Setting $\kappa(m)=\frac{m(m+3)}{4}$ we deduce
\[
u'(z)=e^{-\frac{m\lambda^2}{2}}\frac{d}{dz}\left(\, e^{\frac{m\lambda^2}{2(1-z)}}(1-z)^{-\kappa(m)}\,\right)
\]
\[
=e^{-\frac{m\lambda^2}{2}}\left(\frac{m\lambda^2}{2}e^{\frac{m\lambda^2}{2(1-z)}}(1-z)^{-\kappa(m)-2}+\kappa(m)e^{\frac{m\lambda^2}{2(1-z)}}(1-z)^{-\kappa(m)-1}\right).
\]
 Thus
 \[
 u'(0)= \frac{m\lambda^2}{2}+\kappa(m).
 \]
 We set
 \[
 \frac{1}{2}A_1(\lambda)=\frac{\lambda}{2}\bigl( \,u'(0)+\frac{3}{2}\bigr)=\frac{\lambda}{2}\Bigl( \,\frac{m\lambda^2}{2}+\kappa(m)+\frac{3}{2}\,\Bigr).
 \]
 Similarly, we deduce
 \[
 u''(0)=\frac{m\lambda^2}{2}\left(\frac{m\lambda^2}{2}+\kappa(m)+2\right)+\kappa(m)\Bigl(\,\frac{m\lambda^2}{2}+\kappa(m)+1\,\Bigr)
 \]
 \[
 =\underbrace{\frac{m^2\lambda^4}{4}+(\kappa(m)+1)m\lambda^2+\kappa(m)\bigl(\,\kappa(m)+1\,\bigr)}_{=:\frac{1}{2}A_0(\lambda)}.
 \]
 We set  $A_2(\lambda):=\frac{3}{2}\lambda^2$. We have
 \[
 2\Psi_{zz}(\lambda,0)=A_2(\lambda)\Psi''_{\lambda\lambda}(\lambda,0)+A_1(\lambda)\Psi'_\lambda(\lambda,0)+A_0(\lambda)\Psi(\lambda,0).
 \]
 Using (\ref{sq}) and (\ref{fy4}) we deduce
 \[
 \bsE_S\bigl[\, q(A) f(A)\,\bigr]=\frac{1}{\sqrt{\pi}}\int_\bR\Bigl( \,A_2(\lambda)\Psi''_{\lambda\lambda}(\lambda,0)+A_1(\lambda)\Psi'_\lambda(\lambda,0)+A_0(\lambda)\Psi(\lambda,0)\,\Bigr) e^{-\lambda^2} d\lambda
 \]
 \[
 =\frac{1}{\sqrt{\pi}}\int_\bR\Psi(\lambda,0))\Biggl(\frac{d^2}{d\lambda^2}\bigl(\,A_2(\lambda) e^{-\lambda^2}\,\bigr)-\frac{d}{d\lambda}\bigl(\,A_1(\lambda) e^{-\lambda^2}\,\bigr)+A_0(\lambda) e^{-\lambda^2}\,\Biggr) d\lambda.
 \]
 \[
 =\frac{1}{\sqrt{\pi}}\int_\bR P_{4,m}(\lambda) \Phi(\lambda)^2  \Psi(\lambda,0) e^{-\lambda^2}\,d\lambda,
 \]
 where
 \[
 P_{4,m}(\lambda)= A''_2(\lambda)-4\lambda A_2'(\lambda) +4\lambda^2 A_2(\lambda)-A_1'(\lambda)+2\lambda A_1(\lambda)+A_0(\lambda)
 \]
 \[
 = C_4(m)\lambda^4+C_2(m)\lambda^2+C_0(m) ,
 \]
 where the coefficients $C_0(m), C_2(m), C_4(m)$ are polynomials  in $m$.  Recalling that 
 \[
 \Psi(\lambda,0)=\bsE_G\bigl[\,|\det(\lambda+B)|\,\bigr] \bsC_m\rho_{m+1,1/2}(\lambda),
 \]
 we deduce
 \[
 \bsE_S\bigl[\, q(A) f(A)\,\bigr]= \frac{C_4(m)}{\sqrt{\pi}} \int_\bR\rho_{m+1,\frac{1}{2}}(\lambda) \lambda^4 \frac{e^{-\lambda^2}}{\sqrt{\pi}} d\lambda+\frac{C_2(m)}{\sqrt{\pi}} \int_\bR\rho_{m+1,\frac{1}{2}}(\lambda) \lambda^2 \frac{e^{-\lambda^2}}{\sqrt{\pi}} d\lambda 
 \]
 \[
 +\frac{C_0(m)}{\sqrt{\pi}} \int_\bR\rho_{m+1,\frac{1}{2}} (\lambda) \frac{e^{-\lambda^2} }{\sqrt{\pi}}d\lambda.
 \]
Using (\ref{asyrho}) with $v=1/2$  we deduce that as $m\to \infty$ we have
\[
\bsE_S\bigl[\, q(A) f(A)\,\bigr]\sim \bsC_m m^{-1/2}2^{\frac{1}{2}}\Biggl(  2^{-1} \frac{3C_4(m)}{\sqrt{\pi}} + 2^{-1/2}\frac{C_2(m)}{\sqrt{\pi}} +\frac{C_0(m)}{\sqrt{\pi}}\,\Biggr).
\]
Upon investigating the definition of $A_0(\lambda)$, $A_1(\lambda)$, and $A_0(\lambda)$  we see that of the three  
\[
\deg C_0(m)=4> \deg C_2(m), \deg C_4(m).
\]
The degree-$4$ term in $C_0(m)$ comes from the product 
\[
2\kappa(m)(\kappa(m)+1)=\frac{m^4}{2}+\mbox{ lower order terms}.
\]
We conclude that as $m\to\infty$ we have
\[
\bsE_S\bigl[\, q(A) f(A)\,\bigr] \sim\frac{\bsC_m}{\sqrt{2\pi}}\; m^{\frac{7}{2}}.
\]
\end{proof}

To understand the   2nd chaos component of  $|\det A|$ we need to also understand  the inner product  in $L^2(\eS_m^v)_2^{inv}$. For simplicity will write $\bsE$ instead of the more precise $\bsE_{\eS_m^v}$.

We know that
\[
\bsE\bigl[\,p(A)\,\bigl]=\bigl[ q(A)\,\bigr]=m(m+2)v.
\]
This implies that
\[
\bsE\bigl[\, p(A)^2\,\bigr]=\bsE\bigl[ (\tr A)^4\,\bigl]=3m^2(m+2)^2 v^2.
\]
To compute  $\bsE[\,p(A)q(A)\,]$, $\bsE[\,q(A)^2\,]$ we will use Wick's formula, \cite[Thm. 1.28]{Jan}.  We have
\[
p(A)q(A)=\left(\sum_i a_{ii}^2+2\sum_{i<j}a_{ii}a_{jj}\right)\left(\sum_k a_{kk}^2+2\sum_{k<\ell}a_{k\ell}^2\right)
\]
\[
=\underbrace{\sum_i a_{ii}^4}_{S_1}+\underbrace{2\sum_{i<k}a_{ii}^2a_{kk}^2}_{S_2}+\underbrace{2\sum_{i,\;k<\ell}a_{ii}^2 a_{k\ell}^2}_{S_3}+\underbrace{2\sum_{k, \;i<j}a_{kk}^2a_{ii}a_{jj}}_{S_4}+\underbrace{4\sum_{i<j,\;k<\ell}a_{ii}a_{jj}a_{k\ell}^2}_{S_5}.
\]
We have
\[
\bsE[S_1]=\bsE\left[ \sum_i a_{ii}^4\right]= m\bsE[a_{11}^4]= 27 mv^2.
\]
\[
\bsE[S_3]=\bsE\left[ 2\sum_{i,\;k<\ell}a_{ii}^2 a_{k\ell}^2\right]= m^2(m-1)\bsE[a_{11}^2]\bsE[a_{12}^2]= 3m^2(m-1)v^2.
\]
\[
\bsE[S_5]=\bsE\left[\,4\sum_{i<j,\;k<\ell}a_{ii}a_{jj}a_{k\ell}^2\,\right] =m^2(m-1)^2\bsE[a_{11}a_{22}]\bsE[a_{12}^2]=m^2(m-1)^2v^2.
\]
\[
\bsE[S_2]=\bsE\left[\, 2\sum_{i<k}a_{ii}^2a_{kk}^2\,\right]=m(m-1)\bsE[a_{11}^2 a_{22}^2]
\]
Using Wick's formula we deduce
\begin{equation}\label{a11a22}
\bsE[a_{11}^2 a_{22}^2]=\bsE[a_{11}^2]\bsE[ a_{22}^2]+2\bsE[a_{11}a_{22}]^2=11v^2.
\end{equation}
Hence
\[
\bsE[S_2]=11m(m-1) v^2.
\]
\[
\bsE[S_4]=\bsE\left[ 2\sum_{k, \;i<j}a_{kk}^2a_{ii}a_{jj}\right]
\]
\[
=2\bsE\left[\sum_{i<j} a_{ii}^3a_{jj}\right] +2\bsE\left[\sum_{i<j} a_{ii}a_{jj}^3\right]+2\bsE\left[\sum_{i<j,\;\;k\neq i,j}a_{kk}^2a_{ii}a_{jj}\,\right]
\]
\[
=4\bsE\left[\sum_{i<j} a_{ii}^3a_{jj}\right]+m(m-1)(m-2)\bsE[a_{11}a_{22}a_{33}^2]
\]
\[
=2m(m-1)\bsE[a_{11}^3a_{22}]++m(m-1)(m-2)\bsE[a_{11}a_{22}a_{33}^2].
\]
Using Wick's formula we deduce
\[
\bsE[a_{11}^3a_{22}]=3\bsE[a_{11}^2]\bsE[a_{11}a_{22}]=9v^2
\]
\[
\bsE[a_{11}a_{22}a_{33}^2]=\bsE[a_{11}a_{22}]\bsE[a_{33}^2]+2\bsE[a_{11}a_{22}]^2=5v^2,
\]
Hence
\[
\bsE[S_4]=18m(m-1)v^2+5m(m-1)(m-2)v^2=m(m-1)(5m+8)v^2 .
\]
We have
\[
q(A)^2=\left(\, \underbrace{\sum_i a_{ii}^2}_{X}\; + \; \underbrace{2\sum_{k<\ell} a_{k\ell}^2 }_{Y}\,\right)^2= X^2+Y^2+2XY.
\]
The random variables $X$ and $Y$ are independent and thus
\[
\bsE\bigl[\, q(A)^2\,\bigr]=\bsE[X^2]\;+\;\bsE[Y^2]+2\bsE[X]\bsE[Y].
\]
We have
\[
\bsE[X]=3mv,\;\;\bsE[Y]= m(m-1)v,\;\;2\bsE[XY]= 6m^2(m-1)v^2.
\]
Next,
\[
X^2=\sum_i a_{ii}^4 +2\sum_{I<j} a_{ii}^2a_{jj}^2,
\]
\[
\bsE[X^2]=m\bsE[a_{11}^4]+m(m-1)\bsE[a_{11}^2a_{22}^2]\stackrel{(\ref{a11a22})}{=}27mv^2+11m(m-1)v^2=11m^2v^2+16mv^2,
\]
\[
Y^2=4 \left(\sum_{k<\ell} a_{k\ell}^2\right)^2=4\sum_{k<\ell}a_{k\ell}^4+8\sum_{\substack{ i<j,\;\;k<\ell\\ (i,j)\neq(k,\ell)}}a_{ij}^2a_{kl}^2\,
\]
\[
\bsE[Y^2]=4\binom{m}{2}\bsE[a_{12}^4]+8\binom{\binom{m}{2}}{2} \bsE[a_{12}^2]^2.
\]
\[
=6m(m-1)v^2+ 8\binom{m}{2}\left(\,\binom{m}{2}-1\right) v^2=m(m-1)v^2\bigl(\,6 +2(m+1)(m-2)\,\bigr).
\]
We summarize the results we have obtained so far.  Below we denote by $o(1)$ a function of $m$, independent of $v$ that goes to $0$ as $m\to\infty$.
\[
\bsE\bigl[\,p(A)\,\bigr]=m(m+2)v,
\]
\[
\bsE\bigl[\,p(A)^2\,\bigr]=3m^2(m+2)^2v^2=3m^4v^1(1+o(1)),
\]
\[
\bsE[q(A)]=m(m+2)v,
\]
\[
\bsE\bigl[\, q(A)^2\,\bigr]=mv^2(2\,{m}^{3}+2\,{m}^{2}+9\,m+14)=2m^4v^2(1+o(1)),
\]
\[
E[p(A)q(A)]= \bigl({m}^{3}+3{m}^{2}+12{m}+11\bigr)mv^2=m^4v^2(1+o(1)).
\]
We have
\[
\bsE[\bar{p}(A)^2]= \bsE[p(A)^2] -2m(m+2)\bsE[p(A)] +m(m+2)v=2m^4v^2(1+o(1)).
\]
\[
\bsE[\,\bar{p}(A)\bar{q}(A)\,]= E[p(A)q(A)] -m^2(m+2)^2v^2=-m^3v^2(1+o(1)),
\]
\[
\bsE[\bar{q}(A)^2]= mv^4v^2(1+o(1)).
\]
Thus, in the basis $\bar{p}(A),\bar{q}(A)$ of $L^2(\eS_m^v)^{inv}_2$ the inner product is given by the symmetric matrix
\[
Q_m=m^4v^2\left[
\begin{array}{cc}
2 & o(1)\\
o(1) & 1
\end{array}
\right].
\]
This proves that  the component of $f(A)$ in $L^2(\eS_m^v)_2^{inv}$ has a decomposition
\[
f_2(A) =x_m \bar{p}(A)+y_m\bar{q}_m q(A),
\]
where, as $m\to \infty$
\[
x_m \sim \frac{1}{2m^4v^2} \Bigl( \bsE_{\eS_m^v}\bigl[\,p(A) f(A)\,\bigr]-m(m+2)v \bsE_{\eS^v_m}\bigl[\, f(A)\,\bigr]\,\Bigr)
\]
\[
\sim  \frac{(2v)^{\frac{m+2}{2}}}{2m^4v^2} \Bigl( \bsE_{\eS_m^{1/2}}\bigl[\,p(A) f(A)\,\bigr]-\frac{m(m+2)}{2} \bsE_{\eS^v_m}\bigl[\, f(A)\,\bigr]\,\Bigr),
\]
\[
y_m \sim \frac{1}{m^4v^2}\Bigl( \bsE_{\eS_m^v}\bigl[\,p(A) f(A)\,\bigr]-m(m+2)v \bsE_{\eS_m^v}\bigl[\, f(A)\,\bigr]\,\Bigr),
\]
\[
\sim  \frac{(2v)^{\frac{m+2}{2}}}{2m^4v^2} \Bigl( \bsE_{\eS_m^{1/2}}\bigl[\,q(A) f(A)\,\bigr]-\frac{m(m+2)}{2} \bsE_{\eS^v_m}\bigl[\, f(A)\,\bigr]\,\Bigr).
\]
Using (\ref{fainf}),(\ref{pfainf}) and (\ref{qfainf}) we deduce that there exist two universal constant $z_1,z_2$,independent of $m$ and $v$ such that, as $m\to\infty$,
\begin{equation}\label{xym}
x_m\sim z_1\bsC_mv^{\frac{m-2}{2}}m^{-5/2},\;\;y_m\sim z_2\bsC_m  v^{\frac{m-2}{2}} m^{-1/2}.
\end{equation}
In the problem investigated in this paper the variance $v$ also depends on $m$, $v=h_m$.  Recall that the constant $\bsC_m$ grows really fast as $m\to \infty$
\[
\log \bsC_m\sim \frac{1}{2}m\log m.
\]

\begin{proposition}\label{prop: j2} The   Gaussian vector 
\[
J_2(X):=\bigl(\, X(0),\nabla X(0),\nabla^2 X(0)\,\bigr).
\]
 is nondegenerate.

\end{proposition}

\begin{proof}  We set $H:=\nabla^2(0)$ and we denote by $H_{ij}$  its entries. The equality  (\ref{covx4}) shows that $H\in \eS_m^{h_m}$ is a  centered Gaussian random  real symmetric matrix   whose statistic is defined by the equalities
\[
\bsE\bigl[\, H_{ii}^2\,\bigr]=3 h_m,\;\;\bsE\bigl[H_{ii}H_{jj}\,\bigr]=\bsE[\,H_{ij}^2\,\bigr]=h_m,\;\;\forall i\neq j,
\]
while all the other  covariances are trivial. This shows that  the second jet $J_2(X)$ is the direct sum  of mutually independent Gaussian vectors, $J_2(X)=  A\oplus H_0\oplus D$, where  $D=\nabla X(0)$, $H_0$ is the vector with independent  entries $(H_{ij})_{i<j}$ and $A$ is the vector
\[
A= \bigl(\,X(0),H_{11},\dotsc, H_{mm}\,\bigr).
\]
The components $H_0$ and $D$ are obviously  nondegenerate  Gaussian vectors. Thus, the jet  $J_2(X)$ is nondegenerate if and only if the component $A$ is.  The covariance  matrix of $A$ is  $R_m(s_m,d_m,h_m)$ where for any $s,d,h>0$ we denote by $R_m(s,d,h)$   symmetric $(m+1)\times (m+1)$ matrix with entries
\[
r_{00}=s,\;\;r_{0i}=- d,\;\;\forall i=1,\dotsc, m,\;\; r_{ii}=3h,\;\;r_{ij}=h,\;\;\forall 1\leq i<j\leq m.
\]
Note that multiplying the first row by $s^{-1/2}$ and then the first column  by $s^{-1/2}$ we deduce
\[
\det R_m(s,h,d)=s\det R_m(1,\bar{d}, h),\;\;\bar{d}=ds^{-1/2}.
\]
If we add the first column multiplied by $\bar{d}$ to the other columns  we deduce that
\[
\det R_m(1,\bar{d}, h)= \det G_m\bigl(\,3h-\bar{d}^2, h-\bar{d}^2\,\bigr),
\]
where $G_m(a,b)$ denotes the symmetric $m\times  m$ matrix whose diagonal entries  are equal to $a$, and the off diagonal entries equal to $b$. As explained in \cite[Appendix B]{Ni2015}, we have
\[
\det G_m(a)=(a-b)^{m-1}\bigl(\,a+(m-1)b\,\bigr).
\]
Thus
\[
\det R_m(s,d,h)=s (2h)^{m-1}\bigl( \,(m+2)h -m\bar{d}^2\,\bigr)=(2h)^{m-1}\bigl( \,(m+2)hs -m{d}^2\,\bigr).
\]
Thus $J_2(X)$ is nondegenerate if and only if $\frac{h_ms_m}{d_m^2}\neq \frac{m}{m+2}$. Using (\ref{sdh}) we deduce that
\[
\frac{h_ms_m}{d_m^2}=\frac{m}{m+2}\frac{I_{m-1}(w)I_{m+3}(w)}{I_{m+1}(w)^2}.
\]
From the Cauchy  inequality we deduce that $I_{m+1}(w)^2\leq I_{m-1}(w)I_{m+3}(w)$. We cannot have equality because the functions $\sqrt{w(r)}\;r^{\frac{m-1}{2}}$ and $\sqrt{w(r)}\; r^{\frac{m+3}{2}}$ are linearly independent. 
\end{proof}


\begin{thebibliography}{XXXXX}

\bibitem{AN} R. J.  Adler, G. Naizat: {\sl  A central limit theorem for the Euler integral of a Gaussian random field},  \href{http://arxiv.org/abs/1506.08772}{\sf arXiv: 1506.08772}.

\bibitem{AT}  R. J. Adler, J. E. Taylor: {\sl Random Fields and Geometry}, Springer Verlag, 2007.

\bibitem{AGZ} G. W. Anderson, A. Guionnet, O. Zeitouni: {\sl An Introduction to Random Matrices}, Cambridge University Press,  2010.

\bibitem{Arc}  M. Arcones: {\sl  Limit theorems for nonlinear functionals of a stationary Gaussian sequence of vectors}, Ann. of Probability, {\bf 22}(1994), 2242-2274.

\bibitem{AzDaLe} J.-M. Aza\"{i}s,  F. Dalmao, J.R. Le\'{o}n: {\sl CLT for the zeros of Classical Random Trigonometric Polynomials}, Ann. Inst. H.Poincar\'{e}, to appear, \href{http://arxiv.org/abs/1401.5745}{\textsf{arXiv: 1401.5745}}

\bibitem{AL} J.-M. Aza\"{i}s, J. R. Le\'{o}n: {\sl CLT for crossings of random trigonometric polynomials}, Electron. J. Probab. {\bf 18}(2013), no. 68, 1-17.


\bibitem{AzWs} J.-M. Aza\"{i}s, M. Wschebor: {\sl  Level Sets and Extrema of Random Processes}, John Wiley \& Sons, 2009.

\bibitem{BrMaj}  P. Breuer, P. Major: {\sl Central limit theorems  for non-linear functionals of Gaussian fields}, J. of Multivariate Anal., {\bf 13}(1983), 425-441.


\bibitem{ChaSl} D. Chambers, E. Slud: {\sl  Central limit theorems for nonlinear functionals of stationary Gaussian processes}, Probab. Th. Rel. Fields {\bf 80}(1989), 323-346.

\bibitem{Cuz76} J. Cuzik: {\sl A central limit theorem for the number of zeros  of a stationary Gaussian process}, Ann. Probab. {\bf 4}(1976),  547-556.

\bibitem{Dal} F. Dalmao:  {\sl Asymptotic variance and CLT for the number of zeros of Kostlan Shub Smale random polynomials}, Comptes Rendus Mathematique, to appear,  \href{http://arxiv.org/abs/1504.05355}{\textsf{arXiv: 1504.05355}}


\bibitem{DG} P. Deift,  D. Gioev: {\sl Random Matrix Theory: Invariant Ensembles and Universality},  Courant Lecture Notes, vol. 18, Amer. Math. Soc.,  2009.

\bibitem{EL}  A. Estrade, J. R. Le\'{o}n: {\sl  A central limit theorem for the Euler characteristic of a Gaussian excursion set},  Ann.  of Probability, to appear. MAP5 2014-05. 2015. \href{https://hal.archives-ouvertes.fr/hal-00943054v3}{hal-00943054v3}.

\url{https://hal.archives-ouvertes.fr/hal-00943054v3}


\bibitem{Fy} Y. V. Fyodorov: {\sl Complexity of random energy landscapes, glass transition, and absolute value of the spectral determinant of random matrices},  Phys. Rev. Lett,  {\bf 92}(2004), 240601; Erratum: {\bf 93}(2004), 149901.

\bibitem{GW} A. Granville, I. Wigman: {\sl The distribution of zeroes of random trigonometric polynomials}, Amer. J. Math. {\bf 133}(2011), 295-357. \href{http://front.math.ucdavis.edu/0809.1848}{\sf arXiv: 0809.1848}

\bibitem{Jan} S. Janson: {\sl Gaussian Hilbert Spaces}, Cambridge Tracts in Mathematics, vol. 129, Cambridge University Press, 1997.

\bibitem{KL1997}  M. Kratz,  J. R. Le\'{o}n:  {\sl Hermite polynomial expansion for non-smooth functionals of stationary Gaussian processes: crossings and extremes}, Stoch. Proc. Appl. {\bf 77}(1997), 237-252.

\bibitem{KL2001} M. Kratz, J. R. Le\'{o}n:  {\sl Central limit theorems for level functionals of stationary Gaussian processes and fields}, J. Theor. Probab. {\bf 14}(2001), 639-672.



\bibitem{Maj} P. Major:  {\sl Multiple Wiener-Ito integrals}, Lect. Notes in Math., vol. 849, Springer Verlag,  1981.

\bibitem{Mal69} T. Malevich: {\sl Asymptotic normality of the number of crossings of level $0$ by a Gaussian process}, Theory Probab. Appl. {\bf 14}(1969), 287-295.

\bibitem{Mal} P. Malliavin: {\sl Integration and Probability}, Grad. Texts. in Math., vol. 157, Springer Verlag, 1995.



\bibitem{Ni2014} L.I. Nicolaescu: {\sl  Complexity of random smooth functions on compact manifolds}, Indiana J. Math. {\bf 63}(2014), 1037-1065.

\bibitem{Ni2015} L.I. Nicolaescu: {\sl Critical sets of random smooth functions on compact manifolds}, Asian J. Math.,  {\bf 19}(2015), 391-432.

\bibitem{Nivar} L.I. Nicolaescu: {\sl  Critical points of  multidimensional random Fourier series: variance estimates}, \href{http://arxiv.org/pdf/1310.5571.pdf}{\textsf{arXiv: 1310.5571}}

\bibitem{Nispec} L.I. Nicolaescu: {\sl Random Morse functions and  spectral geometry}, \href{http://arxiv.org/abs/1209.0639}{\sf arXiv: 1209.0639}.

\bibitem{NP} I. Nourdin, G. Peccati: {\sl  Normal Approximations with Malliavin Calculus. From Stein's Method to Universality}, Cambridge Tracts in Math., vol.192, Cambridge University Press, 2012.

\bibitem{NPP} I. Nourdin, G. Peccati, M. Podolskij: {\sl  Quantitative Breiuer-Major theorems}, Stoch.Processes and Appl., {\bf 121}(2011), 793-811.

\bibitem{Nua} D. Nualart: {\sl  The  Malliavin Calculus and Related Topics}, 2nd Edition, Springer Verlag, 2006.

\bibitem{NuaPe} D. Nualart, G. Peccati: {\sl  Central limit theorems for sequences of multiple stochastic integrals}, Ann. Probab. {\bf 33}(2005), 177-193.

\bibitem{Slud91} E. Slud: {\sl Multiple Wiener-It\^{o} expansions for level-crossing-count functionals}, Prob. Th. Rel. Fields, {\bf 87}(1991), 349-364.

\bibitem{Slud94} E. Slud: {\sl MWI representation of the number of curve-crossings by a differentiable Gaussian process with applications}, Ann. Probab. {\bf 22}(1994), 1355-1380.



\bibitem{ST} M. Sodin, B. Tsirelson:{\sl  Random complex zeroes, I. Asymptotic normality}, Israel J. Math., {\bf 144}, 125-149.



\end{thebibliography}
\end{document}